\documentclass[sn-mathphys,Numbered]{sn-jnl}

\usepackage{graphicx}%
\usepackage{multirow}%
\usepackage{amsmath,amssymb,amsfonts}%
\usepackage{amsthm}%
\usepackage{mathrsfs}%
\usepackage[title]{appendix}%
\usepackage{xcolor}%
\usepackage{textcomp}%
\usepackage{manyfoot}%
\usepackage{booktabs}%
\usepackage{algorithm}%
\usepackage{algorithmicx}%
\usepackage{algpseudocode}%
\usepackage{listings}%
\usepackage{placeins}
\numberwithin{equation}{section}

\theoremstyle{thmstyleone}%
\newtheorem{theorem}{Theorem}
\theoremstyle{thmstyletwo}%
\newtheorem{remark}{Remark}%
\newtheorem{lemma}[theorem]{Lemma}%

\newtheorem{proposition}[theorem]{Proposition}

\theoremstyle{thmstylethree}%

\begin{document}

\title[Article Title]{Maximum principle for optimal control of infinite horizon stochastic difference equations driven by fractional noises}

\author[1]{\fnm{Yuecai} \sur{Han}}\email{hanyc@jlu.edu.cn}

\author*[1]{\fnm{Yuhang} \sur{Li}}\email{yuhangl22@mails.jlu.edu.cn}

\equalcont{These authors contributed equally to this work.}

\affil[1]{\orgdiv{School of Mathematics}, \orgname{Jilin University}, \orgaddress{ \city{Changchun}, \postcode{130012},  \state{Jilin Province}, \country{China}}}



\abstract{In this paper, infinite horizon stochastic difference equations and backward stochastic difference equations with fractional noises are studied. The main difficulty comes from fractional noises on infinite horizon. Motivated by discrete-time optimal control problem driven by fractional noises \cite{han2025maximum} and on infinite horizon \cite{ji2024infinite}, the stochastic maximum principle for discrete-time control problem driven by fractional noises in infinite horizon is proved. As an application, an optimal investment problem is solved. }




\maketitle

\section{Introduction}
To solve stochastic optimal control problems, backward stochastic differential equations (BSDEs, in short) has been introduced since last century by Bismut \cite{bismut1973conjugate} for linear case and by Pardoux and Peng \cite{pardoux1990adapted} for general case. Then researchers made much progress on this area, some earlier works refer to \cite{peng1993backward, el1997backward}. In comparison, the backward stochastic difference equations (BS$\Delta$Es, in short) are studied to handle stochastic optimal control problem in discrete-time case. As mentioned in \cite{ji2024infinite}, there are two main types of BS$\Delta$Es, we concern the one that is driven by the increment of a martingale process. 
As we know, the works of this type of BS$\Delta$Es are fewer. 
Cheridito and Stadje \cite{cheridito2013bs} established existence results, comparison principles for solutions of BS$\Delta$Es and the convergence of solution of BS$\Delta$Es to solution of BSDEs.  Bielecki et al. \cite{bielecki2015dynamic} used BS$\Delta$Es in finance to present an arbitrage free theoretical framework. Recently, Ji and Zhang \cite{ji2024infinite} introduced BS$\Delta$Es on infinite horizon and proved the existence and uniqueness results.

Based on BSDEs and BS$\Delta$Es, stochastic maximum principle (SMP, in short) are studied for both continuous time systems and discrete-time systems to obtain the necessary conditions for optimal control. For SMP in continuous time case,  it has been investigated from 1970s \cite{kushner1972necessary,bismut1978introductory}. In 1990, Peng \cite{peng1990general} obtained the SMP for general control systems, where the control domain could be non-convex and control process could appear in the diffusion term. Then researchers made much progress in SMP for continuous control problems, some related works refer to \cite{zhou1998stochastic,han2010maximum,chen2010maximum,li2012stochastic}. For SMP in discrete-time case, it was first given by Lin and Zhang \cite{lin2014maximum}. Then it was investigated in many areas, for example, stochastic games \cite{wu2022maximum}, mean-field case \cite{dong2022maximum}, delayed systems \cite{dong2023maximum}, forward-backward systems \cite{ji2022maximum}, infinite horizon case \cite{ji2024infinite}, systems driven by fractional noises \cite{han2025maximum}.

However, all above works only consider the case that the systems are driven by white noises. To better describe the real-world, the development of SMP for control systems driven by fractional noises are also developed. For continuous time case, Hu and Zhou \cite{hu2005stochastic} and Duncan and Pasik-Duncan \cite{duncan2013linear} considered the linear quadratic control problem driven by fractional noises and obtained the feedback control. Han et al. \cite{han2013maximum} studied the SMP for nonlinear control systems driven by fractional noises. But as we know, only Han and Li \cite{han2025maximum} investigated SMP for control systems driven by fractional noises in discrete-time case.

In this paper, motivated by discrete-time control problem driven by fractional noises \cite{han2025maximum} and on infinite horizon \cite{ji2024infinite}, we focus on following infinite horizon discrete-time control problem with fractional noises.
The state equation is 

\begin{align}\label{1.1}
\left\{\begin{array}{ll}
X_{n+1}=X_n+b(n,X_n,u_n)+\sigma(n,X_n,u_n)\xi_n^H,\,
\\X_0=x,
\end{array}\right.
\end{align}
to minimize the cost function 
\begin{align*}
J(u)=Y_0,
\end{align*}
where $Y$ is given by the following BS$\Delta$Es
\begin{align*}
\left\{\begin{array}{ll}
e^{-\lambda n^\gamma}\left(Y_n+Z_n\eta_n\right)=e^{-\lambda (n+1)^\gamma}\left(Y_{n+1}+f(n+1,X_{n+1},Y_{n+1},Z_{n+1},u_{n+1})\right),
\\\\\lim_{N\to+\infty}\mathbb{E}[e^{-\lambda N^{\gamma}}\|Y_N\|^{2b\delta_1^{-1}...\delta_{N}^{-1}}]=0,
\end{array}\right.
\end{align*}
for some give $\lambda>0,\, b,\gamma>1,\,p_i\in(0,1),\, i\ge 0.$ The significant differences with recursive control problem studied in \cite{ji2024infinite} is that our discount value is $e^{-\lambda n^\gamma}$ rather than $e^{-\lambda n}$ and the terminal condition is more complex. This is because it is unable to show $\mathbb{E}|Y_n|^2\le O(e^{-\lambda n})$ for any $\lambda >0$ and the existence and  uniqueness result of solutions of S$\Delta$Es and BS$\Delta$Es needs stronger conditions in fractional noises case.

We first show the solvability of the S$\Delta$Es (\ref{1.1}), there are two main difficulties and innovations in the proof:
\begin{itemize}
\item As we well know, $X_t=e^{\sigma B_t}$ and $X_t^H=e^{\sigma B_t^H}$ could be solutions to linear stochastic differential equations driven by standard Brownian motion and fractional Brownian motion, respectively. It is easy to show 
\begin{align*}
\mathbb{E}\|X_t\|^p=O(e^{C_pt}),\quad \forall t\ge 0,\, p\ge1,
\end{align*}
for some $C_p$ only depend on $p$. So it is natural to choose norms as
\begin{align*}
\|X_\cdot\|^2_\lambda=\mathbb{E}\int_0^{+\infty}e^{-\lambda t}\|X_t\|^2dt,
\end{align*}
 for continuous case and 
\begin{align*}
\|X_\cdot\|^2_\lambda=\mathbb{E}\sum_{n=0}^{+\infty}e^{-\lambda n}\|X_n\|^2,
\end{align*}
for discrete-time case on infinite horizon. In the same way, the norm for $X_t^H$ should be
\begin{align*}
\|X^H_\cdot\|^2_\lambda=\mathbb{E}\int_0^{+\infty}e^{-\lambda t^{2H}}\|X_t^H\|^2dt,\quad H>\frac{1}{2},
\end{align*}
in continuous case. Surprisingly, we find the following norm is well posedness for S$\Delta$Es (\ref{1.1}) for all $H\in(0,1)$ in discrete-time case:
\begin{align}\label{1.2}
\|X_\cdot\|_{\lambda,\gamma}^2=\sum_{n=0}^{+\infty} e^{-\lambda n^\gamma}\|X_n\|^2,
\end{align}
where $\lambda>0$, $\gamma>1$ do not dependent on $H$ and could be arbitrary closed to $0$ and $1$, respectively. We will show the existence and uniqueness of the solution of S$\Delta$Es (\ref{1.1}) under the norm (\ref{1.2}).

\item The second difficulty caused by the self-dependence of the fractional noises, which leads to 
\begin{align*}
\mathbb{E}[X_n\xi_n^H]\neq0,\quad \mathbb{E}[X_n\xi_n^2]\neq\mathbb{E}[X_n].
\end{align*}
In finite horizon case, Han and Li \cite{han2025maximum} assume $\mathbb{E}\|X_0\|^{2a}<+\infty$ for $a>1$, which could be arbitrary closed to $1$, and prove 
\begin{align*}
\mathbb{E}\|X_n\|^{2a\delta^{n}}<+\infty,\quad n=0,1,...,N,
\end{align*}
step by step for some $\delta\in(0,1)$ related to $a$. However, this method could not be applied in infinite horizon case, because $2a\delta^n\to0$, as $n\to+\infty$. To handle this, we construct a vector $\vec{\delta}=(\delta_1,\delta_2,...)$, such that $\delta_i\in(0,1),\, i\ge1$, and
\begin{align*}
 |\vec{\delta}|:=\lim_{n\to +\infty}\Pi_{i=1}^n\delta_i>\frac{1}{a}.
\end{align*}
Then, through more complex calculus, we show 
\begin{align*}
\mathbb{E}\sum_{n=0}^{+\infty}e^{-\lambda n^\gamma}\|X_n\|^{2a\delta_1...\delta_n}<+\infty,
\end{align*}
which implies
\begin{align*}
\mathbb{E}\sum_{n=0}^{+\infty}e^{-\lambda n^\gamma}\|X_n\|^{2}<+\infty.
\end{align*}
\end{itemize}

To obtain the SMP, we introduce $(p,q)$ as adjoint BS$\Delta$Es as (\ref{BSDEpq}). To show the solvability of the solution of adjoint equation, we prove the existence and uniqueness results for the following BS$\Delta$Es:
\begin{align}
\left\{\begin{array}{ll}
e^{-\lambda n^\gamma}\left(Y_n+Z_n\eta_n\right)=e^{-\lambda(n+1)^\gamma}\left[Y_{n+1}+f(n+1,Y_{n+1},Z_{n+1})\right]\\\\
\qquad\qquad\qquad\qquad\qquad+e^{-\lambda(n+1)^\gamma}g(n+1,Y_{n+1},Z_{n+1})\mathbb{E}^{\mathcal{F}_{n+1}}[\xi_{n+1}^H],
\\\\\lim_{N\to+\infty}\mathbb{E}[e^{-\lambda N^{\gamma}}\|Y_N\|^{2b\delta_1^{-1}...\delta_{N}^{-1}}]=0.
\end{array}\right.
\end{align}
The proof is combining with that of S$\Delta$Es (\ref{1.1}) and section 2 of \cite{ji2024infinite}, in comparison, we show
\begin{align*}
\mathbb{E}\sum_{n=0}^{+\infty}e^{-\lambda n^\gamma}\|Y_n\|^{2b\delta_1^{-1}...\delta_n^{-1}}<+\infty,
\end{align*}
for some $b\ge 1$.

Based on the well-posedness of the solution of S$\Delta$Es (\ref{1.1}) and BS$\Delta$Es, we obtain the SMP for infinite discrete-time control problem with fractional noises. Moreover, an optimal investment problem is solved.

The rest of this paper is organized as follows. In section 2, we introduce the S$\Delta$Es and BS$\Delta$Es
with fractional noise, and prove the existence and the
uniqueness of the solutions of the S$\Delta$Es and BS$\Delta$Es. In section 3, we obtain the SMP for infinite discrete-time control problem driven by  fractional noises and prove the related verification theorem. In section 4,  an optimal investment problem is solved to illustrate the main results.

\section{ Infinite horizon S$\Delta$Es and BS$\Delta$Es}
Let the filtration $\mathbb{F}=(\mathcal{F}_n)_{0\le n\le N}$ defined by $\mathcal{F}_n=\mathcal{F}_0\lor\sigma(\xi^H_0,\xi_1^H,...,\xi_{n-1}^H)$. 
Let $(\Omega,\mathcal{F},\{\mathcal{F}_n\}_{n\in\mathbb
{Z}^+},P)$ be a filtered probability space, $\mathcal{F}_0\subset \mathcal{F}$ be a sub $\sigma$-algebra. $\{\xi_n^H\}_{n\in\mathbb
{Z}^+}$ is a sequence of fractional noises described by the increment of a $m$-dimensional fractional Brownian motion such that  $\xi_n^H=B^H(n+1)-B^H(n)$. 

Denote by $L^\beta(\mathcal{F}_n;\mathbb{R}^d)$, or $L^\beta(\mathcal{F}_n)$ for simplify,  the set of all $\mathcal{F}_n$-measurable  random variables $X$ taking values in $\mathbb{R}^d$, such that $\mathbb{E}\|X\|^\beta<+\infty$. Denote by $L^{\beta,\vec{\delta},\lambda,\gamma}_\mathcal{F}(0,+\infty;\mathbb{R}^d)$, or $L^{\beta,\vec{\delta},\lambda,\gamma}_\mathcal{F}(0,+\infty)$ for simplify, the set of all $\mathcal{F}_n$-adapted process $X=(X_n)_{n\in\mathbb
{Z}^+}$ such that 
\begin{align}\label{norm}
\|X_\cdot\|_{\beta,\vec{\delta},\lambda,\gamma}=\sum_{n=0}^{+\infty}e^{-\lambda n^\gamma}\mathbb{E}\left(\|X_n\|^{\beta \delta_1...\delta_n}\right)<+\infty,
\end{align}
for $\beta,\lambda>1$, where $\vec{\delta}=(\delta_1,...,\delta_n,...)$ satisfies $p_n\in(0,1), n\ge 1$, and $$|\vec{\delta}|:=\lim_{n\to +\infty}\Pi_{i=1}^n\delta_i>0.$$
In this paper, we choose $\delta_n=1-(n+2)^{-\theta}$, $\theta>1$, $n\ge 1$, and denote the corresponding vector as $\mathbf{p}^\theta$. Then we give a lower bound of $|\vec{\delta}^\theta|$.
\begin{proposition}
Let $\delta_n=1-(n+2)^{-\theta}$, $\theta>1$. Then $|\vec{\delta}^\theta|\ge exp\left(\frac{2^{1-\theta}}{1-\theta}+\frac{2^{1-2\theta}}{1-2\theta}\right)$.
\end{proposition}
\begin{proof}
Define $S_N=\Pi_{i=1}^N\delta_i$.
Notice that $ln(1+x)\ge x-\frac{x^2}{2}, x\in(-1,0)$, so we have that
\begin{align*}
lnS_N=&\sum_{i=1}^Nln\delta_i\\
\ge&-\sum_{i=1}^N(i+2)^{-\theta}-\frac{1}{2}\sum_{i=1}^N(i+2)^{-2\theta}\\
\ge& -\int_1^{N+1}(x+1)^{-\theta}dx-\frac{1}{2}\int_1^{N+1}(x+1)^{-2\theta}dx\\
=&\frac{-1}{1-\theta}[(N+2)^{1-\theta}-2^{1-\theta}]-\frac{1}{1-2\theta}[(N+2)^{1-2\theta}-2^{1-2\theta}].
\end{align*}
Then 
\begin{align*}
|\vec{\delta}^\theta|=exp\left(\lim_{N\to+\infty}lnS_N\right)\ge exp\left(\frac{2^{1-\theta}}{1-\theta}+\frac{2^{1-2\theta}}{1-2\theta}\right).
\end{align*}
\end{proof}

\begin{remark}
The above proposition implies that $|\vec{\delta}^\theta|\to1^-$, as $\theta\to+\infty$.
\end{remark}

\begin{remark}
It is easy to check that if $\beta|\vec{\delta}|\ge2$, then equation (\ref{norm}) implies 
\begin{align*}
\|X_\cdot\|^2_{\lambda,\gamma}:=\sum_{n=0}^{+\infty}e^{-\lambda n^\gamma}\mathbb{E}\left(\|X_n\|^2\right)<+\infty.
\end{align*}
We denote $L_\mathcal{F}^{2,\lambda,\gamma}(0,+\infty)$ by the set of all adapted process such that $\|X_\cdot\|_{\lambda,\gamma}<+\infty$.
\end{remark}
Define $\mathcal{T}=\{0,1,...\}$. Then we consider the following S$\Delta$Es,
\begin{align}\label{SDE}
\left\{\begin{array}{ll}
X_{n+1}=X_n+b(n,X_n)+\sigma(n,X_n)\xi_n^H,\, n\ge0,
\\X_0=x.
\end{array}\right.
\end{align}
Here $x\in L^{2a}(\mathcal{F}_0)$ is independent with $\{\xi_n\}$, 
$b(n,x)$ and $\sigma(n,x) $ are measurable functions on $\mathcal{T}\times \mathbb{R}^d$ with values in $\mathbb{R}^d$ and $\mathbb{R}^{d\times m}$, respectively. We give the following assumptions:
\begin{enumerate}
	\item[(H2.1)]  $b$ and $\sigma$ are adapted processes: $b(\cdot,x), \sigma(\cdot,x)\in L^{2a,\mathbf{p}^\theta,\lambda,\gamma}_{\mathcal{F}}(0,+\infty)$,\, $\forall x\in\mathbb{R}^d$, for some $\theta, \lambda>1$, such that $|\vec{\delta}^\theta|\ge a^{-1}$.

\item[(H2.2)] There exists some constants $L>0$, such that 
\begin{align*}
\|x_1+b(n,x_1)-x_2-b(n,x_2)\|&+\|\sigma(n,x_1)-\sigma(n,x_2)\|\le L\|x_1-x_2\|,\\
&\forall n\in\mathcal{T},\quad x_1,x_2\in\mathbb{R}^d.
\end{align*}
\end{enumerate}

Then we show the existence and uniqueness result.

\begin{theorem}\label{theSDE}
Assume  that (H2.1), (H2.2) hold, then S$\Delta$Es (\ref{SDE}) has a unique solution in $L_\mathcal{F}^{2a,\vec{\delta}^\theta,\lambda,\gamma}(0,+\infty)$.
\end{theorem}
\begin{proof}
We first show $\{X_n\}_{n\in\mathcal{T}}$ in $L_{\mathcal{F}}^{2a,\vec{\delta}^\theta,\lambda,\gamma}(0,+\infty)$. For $X_1$, through Jensen inequality  and Hölder inequality, we have that
\begin{align*}
\mathbb{E}\|X_1\|^{2a\delta_1}\le& 2^{2a\delta_1}\left(\mathbb{E}\|x+b(0,x)\|^{2a\delta_1}+\mathbb{E}\|\sigma(0,x)\xi_0^H\|^{2a\delta_1}\right)\\
\le&(2L)^{2a\delta_1}\left(\mathbb{E}\|b(0,0)+x\|^{2a\delta_1}+\left(\mathbb{E}\|\sigma(0,0)+x\|^{2a}\right)^{\delta_1}\times\left(\mathbb{E}|\xi_0^H|^{2a\delta_1(1-\delta_1)^{-1}}\right)^{1-\delta_1}\right)\\
\le& (4L)^{2a\delta_1}\left(\left(\mathbb{E}\|b(0,0)\|^{2a}\right)^{\delta_1}+\left(\mathbb{E}\|\sigma(0,0)\|^{2a}\right)^{\delta_1}+\left(\mathbb{E}\|x\|^{2a}\right)^{\delta_1}\right)\\
&\qquad\times\left(1+\left(\mathbb{E}|\xi_0^H|^{2a\delta_1(1-\delta_1)^{-1}}\right)^{1-\delta_1}\right).
\end{align*}
By Stirling formula, $$\mathbb{E}|\xi_0^H|^m=\frac{2^{m/2}}{\pi^{1/2}}\Gamma(\frac{m+1}{2})\approx\frac{2^{m/2}}{\pi^{1/2}}\sqrt{\frac{4\pi}{m+1}}\left(\frac{m+1}{2e}\right)^{\frac{m+1}{2}} =O\left((m/e)^{m/2}\right),\, m\to+\infty,$$
which means $\left(\mathbb{E}|\xi_0^H|^m\right)^{1/m}=O(m^{1/2}), m\to+\infty$. Denote $C_s=\sup_{m\ge1}\frac{\left(\mathbb{E}|\xi_0^H|^m\right)^{1/m}}{m^{1/2}}$, then we have
\begin{align*}
\mathbb{E}\|X_1\|^{2a\delta_1}\le&(4L)^{2a\delta_1}\left(1+C_s^{2a\delta_1}\sqrt{2a\delta_1(1-\delta_1)^{-1}}\right)\\
&\qquad\times\left(\left(\mathbb{E}\|b(0,0)\|^{2a}\right)^{\delta_1}+\left(\mathbb{E}\|\sigma(0,0)\|^{2a}\right)^{\delta_1}+\left(\mathbb{E}\|x\|^{2a}\right)^{\delta_1}\right).
\end{align*}
In the same way, we derive
\begin{align*}
&\mathbb{E}\|X_{n+1}\|^{2a\delta_1...\delta_{n+1}}\\
\le&(4L)^{2a\delta_1...\delta_{n+1}}\left(1+C_s^{2a\delta_1...\delta_{n+1}}\sqrt{2a\delta_1...\delta_{n+1}(1-\delta_{n+1})^{-1}}\right)\left(\mathbb{E}\|X_n\|^{2a\delta_1...\delta_n}\right)^{\delta_{n+1}}\\
&\qquad\times\left(\left(\mathbb{E}\|b(n,0)\|^{2a\delta_1...\delta_n}\right)^{\delta_{n+1}}+\left(\mathbb{E}\|\sigma(n,0)\|^{2a\delta_1...\delta_n}\right)^{\delta_{n+1}}+\left(\mathbb{E}\|X_n\|^{2a\delta_1...\delta_n}\right)^{\delta_{n+1}}\right)\\
\le&\left(C_1+C_2(n+3)^{\frac{\theta}{2}}\right)\\
&\qquad\times \left(\left(\mathbb{E}\|b(n,0)\|^{2a\delta_1...\delta_n}\right)^{\delta_{n+1}}+\left(\mathbb{E}\|\sigma(n,0)\|^{2a\delta_1...\delta_n}\right)^{\delta_{n+1}}+\left(\mathbb{E}\|X_n\|^{2a\delta_1...\delta_n}\right)^{\delta_{n+1}}\right)\\
\le&C_3(n+3)^{\frac{\theta}{2}}\left(\left(\mathbb{E}\|b(n,0)\|^{2a\delta_1...\delta_n}\right)^{\delta_{n+1}}+\left(\mathbb{E}\|\sigma(n,0)\|^{2a\delta_1...\delta_n}\right)^{\delta_{n+1}}+\left(\mathbb{E}\|X_n\|^{2a\delta_1...\delta_n}\right)^{\delta_{n+1}}\right)\\
\le&C_3(n+3)^{\frac{\theta}{2}}\mathbb{E}\|X_n\|^{2a\delta_1...\delta_n}+C_3(n+3)^{\frac{\theta}{2}}\left(3+\mathbb{E}\|b(n,0)\|^{2a\delta_1...\delta_n}+\mathbb{E}\|\sigma(n,0)\|^{2a\delta_1...\delta_n}\right)
\end{align*}
where $C_1=\max\{(4L)^{2a},(4L)^{2a|\mathbf{p}|}\}$,
\begin{align*}
C_2=\sqrt{2a}\max\{(4L)^{2a},(4L)^{2a|\mathbf{p}|}\}\times\max\{C_s^{2a},C_s^{2a|\mathbf{p}|}\}
\end{align*}
and $C_3=2\max\{C_1,C_2\}$.
Then
\begin{align}\label{2.3}
&e^{-\lambda (n+1)^\gamma}\mathbb{E}\|X_{n+1}\|^{2a\delta_1...\delta_{n+1}}\notag\\
\le& C_3e^{-\lambda\gamma n^{\gamma-1}}(n+3)^{\frac{\theta}{2}}e^{-\lambda n^{\gamma}}\mathbb{E}\|X_n\|^{2a\delta_1...\delta_n}\notag\\
&+C_3e^{-\lambda\gamma n^{\gamma-1}}(n+3)^{\frac{\theta}{2}}e^{-\lambda n^{\gamma}}\left(3+\mathbb{E}\|b(n,0)\|^{2a\delta_1...\delta_n}+\mathbb{E}\|\sigma(n,0)\|^{2a\delta_1...\delta_n}\right).
\end{align}
Choose a $N_0\ge1$, such that $C_3e^{-\lambda\gamma n^{\gamma-1}}(n+3)^{\frac{\theta}{2}}\le\frac{1}{2}$, as $n\ge N_0$. It is easy to check
\begin{align}\label{2.4}
\sum_{n=0}^{N_0}e^{-\lambda n^\gamma}\mathbb{E}\left(\|X_n\|^{2a\delta_1...\delta_n}\right)<+\infty.
\end{align}
Then, through equation (\ref{2.3}), we obtain
\begin{align*}
&\sum_{n=N_0+1}^{+\infty}e^{-\lambda n^\gamma}\mathbb{E}\left(\|X_n\|^{2a\delta_1...\delta_n}\right)\notag\\
\le& \frac{1}{2}\sum_{n=N_0+1}^{+\infty}e^{-\lambda n^\gamma}\mathbb{E}\left(\|X_n\|^{2a\delta_1...\delta_n}\right) +\frac{1}{2}e^{-\lambda N_0^\gamma}\mathbb{E}\left(\|X_{N_0}\|^{2a\delta_1...\delta_{N_0}}\right)\notag\\
&+\frac{1}{2}\sum_{n=N_0}^{+\infty}e^{-\lambda n^{\gamma}}\left(3+\mathbb{E}\|b(n,0)\|^{2a\delta_1...\delta_n}+\mathbb{E}\|\sigma(n,0)\|^{2a\delta_1...\delta_n}\right),
\end{align*}
i.e.,
\begin{align}\label{2.5}
&\sum_{n=N_0+1}^{+\infty}e^{-\lambda n^\gamma}\mathbb{E}\left(\|X_n\|^{2a\delta_1...\delta_n}\right)\notag\\
\le& e^{-\lambda N_0^\gamma}\mathbb{E}\left(\|X_{N_0}\|^{2a\delta_1...\delta_{N_0}}\right)\notag\\
&+\sum_{n=N_0}^{+\infty}e^{-\lambda n^{\gamma}}\left(3+\mathbb{E}\|b(n,0)\|^{2a\delta_1...\delta_n}+\mathbb{E}\|\sigma(n,0)\|^{2a\delta_1...\delta_n}\right)\notag\\
<&+\infty.
\end{align}
Equations (\ref{2.4}) and (\ref{2.5}) show $\{X_n\}_{n\in\mathcal{T}}\in L_{\mathcal{F}}^{2a,\vec{\delta}^\theta,\lambda,\gamma}(0,+\infty)$.

Then we prove uniqueness. Let $\{X_n\}_{n\in\mathcal{T}}$ and $\{\Tilde{X}_n\}_{n\in\mathcal{T}}$ be two solutions of S$\Delta$Es (\ref{SDE}). Similar to the proof of existence, we have
\begin{align*}
&e^{-\lambda (n+1)^\gamma}\mathbb{E}\|\Tilde{X}_{n+1}-X_{n+1}\|^{2a\delta_1...\delta_{n+1}}\notag\\
\le& C_3e^{-\lambda\gamma n^{\gamma-1}}(n+3)^{\frac{\theta}{2}}e^{-\lambda n^{\gamma}}\left(\mathbb{E}\|\tilde{X}_n-X_n\|^{2a\delta_1...\delta_n}\right)^{\delta_{n+1}}.
\end{align*}
Since $\mathbb{E}\|\tilde{X}_0-X_0\|^{2a}=0$, we have
\begin{align*}
\sum_{n=0}^{N_1}e^{-\lambda n^\gamma}\mathbb{E}\left(\|\tilde{X}_n-X_n\|^{2a\delta_1...\delta_n}\right)=0,
\end{align*}
for a fixed $N_1\ge N_0$. For $n>N_1$,
\begin{align*}
&\sum_{n=N_0+1}^{+\infty}e^{-\lambda n^\gamma}\mathbb{E}\left(\|X_n\|^{2a\delta_1...\delta_n}\right)\notag\\
\le& \frac{1}{2}\sum_{n=N_1+1}^{+\infty}e^{-\lambda n^\gamma}\mathbb{E}\left(\|X_n\|^{2a\delta_1...\delta_n}\right) +\frac{1}{2}e^{-\lambda N_1^\gamma}\mathbb{E}\left(\|X_{N_1}\|^{2a\delta_1...\delta_{N_1}}\right)+\frac{1}{2}\sum_{n=N_1}^{+\infty}e^{-\lambda n^{\gamma}},
\end{align*}
i.e.,
\begin{align*}
\sum_{n=N_1+1}^{+\infty}e^{-\lambda n^\gamma}\mathbb{E}\left(\|X_n\|^{2a\delta_1...\delta_n}\right)\le\sum_{n=N_1}^{+\infty}e^{-\lambda n^{\gamma}}.
\end{align*}
Let $N_1\to +\infty$, it follows  $\|X_\cdot\|_{2a,\vec{\delta}^\theta,\lambda,\gamma}\to0$. So that S$\Delta$Es (\ref{SDE}) has a unique solution in $L_\mathcal{F}^{2a,\vec{\delta}^\theta,\lambda,\gamma}(0,+\infty)$.
\end{proof}

Denote by $\hat{L}^{\beta,\vec{\delta},\lambda,\gamma}_\mathcal{F}(0,+\infty;\mathbb{R}^d)$, or $\hat{L}^{\beta,\vec{\delta},\lambda,\gamma}_\mathcal{F}(0,+\infty)$ for simplify, the set of all $\mathcal{F}_n$-adapted process $Y=(Y_n)_{n\in\mathbb
{Z}^+}$ such that 
\begin{align}\label{norm}
\|Y_\cdot\|_{\hat{L}^{\beta,\vec{\delta},\lambda,\gamma}_\mathcal{F}}=\sum_{n=0}^{+\infty}e^{-\lambda n^\gamma}\mathbb{E}\left(\|Y_n\|^{\beta \delta_1^{-1}...\delta_n^{-1}}\right)<+\infty.
\end{align}

Then we introduce the following BS$\Delta$Es:
\begin{align}\label{BSDE}
\left\{\begin{array}{ll}
e^{-\lambda n^\gamma}\left(Y_n+Z_n\eta_n\right)=e^{-\lambda(n+1)^\gamma}\left[Y_{n+1}+f(n+1,Y_{n+1},Z_{n+1})\right]\\\\
\qquad\qquad\qquad\qquad\qquad+e^{-\lambda(n+1)^\gamma}g(n+1,Y_{n+1},Z_{n+1})\mathbb{E}^{\mathcal{F}_{n+1}}[\xi_{n+1}^H],
\\\\\lim_{N\to+\infty}\mathbb{E}[e^{-\lambda N^{\gamma}}\|Y_N\|^{2b\delta_1^{-1}...\delta_{N}^{-1}}]=0.
\end{array}\right.
\end{align}
We now give the following assumptions:

\begin{enumerate}
	\item[(H2.3)]  $f$ and $g$ are adapted processes: $f(\cdot,y,z), \sigma(\cdot,y,z)\in \hat{L}^{2b,\vec{\delta}^\theta,\lambda,\gamma}_{\mathcal{F}}(0,+\infty)$,\, $\forall y,z\in\mathbb{R}^d$, for some $\theta,\gamma>1,\,\lambda>0$.

\item[(H2.4)] There exists some constants $L>0$, such that 
\begin{align*}
\|y_1+f(n,y_1,z_1)-y_2-f(n,y_2,z_2)\|&+\|g(n,y_1,z_1)-g(n,y_2,z_2)\|\\
\le L(\|y_1-y_2\|&+\|z_1-z_2\|),\quad
\forall n\in\mathcal{T},\quad y_1,y_2,z_1,z_2\in\mathbb{R}^d.
\end{align*}

\item[(H2.5)] There exists functions $f_1, g_1: \Omega\times\mathcal{T}\times \mathbb{R}^d\to\mathbb{R}^d$, such that 
\begin{align*}
\lim_{n\to+\infty}\mathbb{E}\left[e^{-\lambda n^{\gamma}}\left(\|f(n,y,z)-f_1(n,y)\|^{2b\delta_1^{-1}...\delta_{n}^{-1}}+\|g(n,y,z)-g_1(n,y)\|^{2b\delta_1^{-1}...\delta_{n}^{-1}}\right)\right]=0,
\end{align*}
and
\begin{align*}
\|f_1(n,y_1)-f_1(n,y_2)\|&+\|g_1(n,y_1)-g_1(n,y_2)\|\le L\|y_1-y_2\|,\quad n\in\mathcal{T},
\end{align*}
for all $y,y_1,y_2,z\in\mathbb{R}^d$.
\end{enumerate}

Then we show the solvability of the BS$\Delta$Es (\ref{BSDE}).

\begin{theorem}\label{BSDE e,u}
Assume that (H2.3) - (H2.5) hold, then BS$\Delta$Es (\ref{BSDE}) has a unique solution pair $\{(Y_n,Z_n)\}_{n\in\mathcal{T}}\in\hat{L}^{2b,\vec{\delta}^\theta,\lambda,\gamma}_{\mathcal{F}}(0,+\infty)\times\hat{L}^{2b,\vec{\delta}^\theta,\lambda,\gamma}_{\mathcal{F}}(0,+\infty)$.
\end{theorem}
\begin{proof}
\textbf{Existence.} For $N>0$, define
\begin{align}\label{BSDE1}
\left\{\begin{array}{ll}
Y_n^N+Z_n^N\eta_n=e^{-\lambda[(n+1)^\gamma-n^\gamma]}\left(Y_{n+1}^N+\Bar{f}^N(n+1,Y_{n+1}^N,Z_{n+1}^N)\right)\\\\
\qquad\qquad\qquad\quad+e^{-\lambda[(n+1)^\gamma-n^\gamma]}\Bar{g}^N(n+1,Y_{n+1}^N,Z_{n+1}^N)\mathbb{E}^{\mathcal{F}_{n+1}}[\xi_{n+1}^H],\, n\le N-1,
\\\\e^{-\frac{1}{2}\lambda N^{\gamma}}\|Y_N^N\|^{b\delta_1^{-1}...\delta_N^{-1}}=0.
\end{array}\right.
\end{align}
Here
\begin{align*}
\bar{f}^N(n,y,z)=\left\{\begin{array}{ll}
f(n,y,z),\quad n\le N-1,\\
f_1(N,y),\quad\quad n=N,
\end{array}\right.
\end{align*}
and
\begin{align*}
\bar{g}^N(n,y,z)=\left\{\begin{array}{ll}
g(n,y,z),\quad n\le N-1,\\
g_1(N,y),\quad\quad n=N.
\end{array}\right.
\end{align*}

Let $\tilde{Y}_n=e^{-\lambda n^{\gamma}}Y_n^N$ and $\tilde{Z}_n=e^{-\lambda n^{\gamma}}Z_n^N$.
Then we rewrite BS$\Delta$Es (\ref{BSDE1}) as
\begin{align}\label{BSDE2}
\left\{\begin{array}{ll}
\tilde{Y}_n+\tilde{Z}_n\eta_n=\tilde{Y}_{n+1}+\tilde{f}(n+1,\tilde{Y}_{n+1}^N,\tilde{Z}_{n+1}^N)+\tilde{g}(n+1,\tilde{Y}_{n+1}^N,\tilde{Z}_{n+1}^N)\mathbb{E}^{\mathcal{F}_{n+1}}[\xi_{n+1}^H]
\\\\\tilde{Y}_N=0,
\end{array}\right.
\end{align}
where
\begin{align*}
\tilde{f}(n,y,z)=e^{-\lambda n^{\gamma}}\bar{f}^N(n,e^{\lambda n^{\gamma}}y,e^{\lambda n^{\gamma}}z),\quad \tilde{g}(n,y,z)=e^{-\lambda n^{\gamma}}\bar{g}^N(n,e^{\lambda n^{\gamma}}y,e^{\lambda n^{\gamma}}z).
\end{align*}
According to (H2.3), (H2.4), it is easy to check $$\mathbb{E}\left[\|\tilde{f}(n,0,0)\|^{2b\delta_1^{-1}...\delta_n^{-1}}+\|\tilde{g}(n,0,0)\|^{2b\delta_1^{-1}...\delta_n^{-1}}\right]<+\infty,$$
and $\tilde{f},\tilde{g}$ are Lipschitz continuous w.r.t $(y,z)$. So that we have the following solvability result for BS$\Delta$Es (\ref{BSDE2}).
\begin{proposition}
Under the assumptions (H2.3)-(H2.5), BS$\Delta$Es (\ref{BSDE2}) has a unique solution pair in $L^2_\mathcal{F}(0,N)\times L^2_\mathcal{F}(0,N-1)$. Moreover, we have
\begin{align*}
\tilde{Y}_n=\mathbb{E}^{\mathcal{F}_n}\left[\tilde{Y}_{n+1}+\tilde{f}(n+1,\tilde{Y}_{n+1}^N,\tilde{Z}_{n+1}^N)+\tilde{g}(n+1,\tilde{Y}_{n+1}^N,\tilde{Z}_{n+1}^N)\xi_{n+1}^H\right],
\end{align*}
and
\begin{align*}
\tilde{Z}_n=\mathbb{E}^{\mathcal{F}_n}\left[\eta_n\cdot\left(\tilde{Y}_{n+1}+\tilde{f}(n+1,\tilde{Y}_{n+1}^N,\tilde{Z}_{n+1}^N)+\tilde{g}(n+1,\tilde{Y}_{n+1}^N,\tilde{Z}_{n+1}^N)\xi_{n+1}^H\right)\right].
\end{align*}
\end{proposition}
\begin{proof}
The proof similar to Theorem 3 of \cite{han2025maximum}, we omit it.
\end{proof}
Then we show $(Y^N,Z^N)$ is a Cauchy sequence. Let $(Y^N,Z^N)$, $(Y^M,Z^M)$ be solutions to BS$\Delta$Es (\ref{BSDE1}) with terminal time $N,M$, respectively, and $N>M\ge \tilde{N_0}$, where $\tilde{N_0}>0$ independent with $M,N,\gamma,\lambda$, which will be defined in the following proof. Define
\begin{align*}
Y^{M, N}=Y^N-Y^M= \begin{cases}0, & n \geq N, \\ Y_n^N, & M \leq n<N, \\ Y_n^N-Y_n^M, & n<M,\end{cases}
\end{align*}
and define $Z^{M,N}$ similarly.

For $n\in[M,N-1]$, recall that $\delta_{n}=1-(n+2)^{-\theta}$ and $q_n=1-p_n$, through Hölder inequality, we have 
\begin{align}\label{2.10}
&e^{-\lambda n^{\gamma}}\left(\mathbb{E}|Y_n^N\|^{2b\delta_1^{-1}...\delta_{n}^{-1}}\right)\notag\\=&\mathbb{E}\left(e^{-\lambda n^{\gamma}}\|Y_n^N\|\right)^{2b\delta_1^{-1}...\delta_{n}^{-1}}\times e^{-\lambda n^{\gamma}(1-2b\delta_1^{-1}...\delta_n^{-1})}\notag\\
\le&Ce^{-\lambda n^{\gamma}(1-2b\delta_1^{-1}...\delta_n^{-1})}\mathbb{E}\left[e^{-\lambda (n+1)^\gamma}\left(\|Y^N_{n+1}\|+\|Z_{n+1}^N\|+\|\bar{f}^N(n+1,0,0)\|\right)\right]^{2b\delta_1^{-1}...\delta_{n}^{-1}}\notag\\
&+Ce^{-\lambda n^{\gamma}(1-2b\delta_1^{-1}...\delta_n^{-1})}\notag\\
&\qquad\times\mathbb{E}\left[e^{-\lambda (n+1)^\gamma}\left(\|Y^N_{n+1}\|+\|Z_{n+1}^N\|+\|\bar{g}^N(n+1,0,0)\|\right)\cdot|\xi_{n+1}^H|\right]^{2b\delta_1^{-1}...\delta_{n}^{-1}}\notag\\
=&Ce^{-\lambda(2b\delta_1^{-1}...\delta_n^{-1}-1) \left((n+1)^\gamma-n^{\gamma}\right)}\notag\\
&\qquad\times e^{-\lambda (n+1)^\gamma}\mathbb{E}\left[\left(\|Y^N_{n+1}\|+\|Z_{n+1}^N\|+\|\bar{f}^N(n+1,0,0)\|\right)\right]^{2b\delta_1^{-1}...\delta_{n}^{-1}}\notag\\
&+Ce^{-\lambda(2b\delta_1^{-1}...\delta_n^{-1}-1) \left((n+1)^\gamma-n^{\gamma}\right)}\notag\\
&\qquad\times e^{-\lambda (n+1)^\gamma}\mathbb{E}\left[\left(\|Y^N_{n+1}\|+\|Z_{n+1}^N\|+\|\bar{g}^N(n+1,0,0)\|\right)\cdot|\xi_{n+1}^H|\right]^{2b\delta_1^{-1}...\delta_{n}^{-1}}\notag\\
\le& Ce^{-\lambda(2b\delta_1^{-1}...\delta_n^{-1}-1) \left((n+1)^\gamma-n^{\gamma}\right)}\cdot\left(1+C_s^{2b\delta_1^{-1}...\delta_{n}^{-1}}(1-\delta_{n+1})^{-\frac{1}{2}}\right)\notag\\
&\quad \times \Bigg\{e^{-\lambda (n+1)^\gamma}\left[\left(\mathbb{E}\|Y^N_{n+1}\|^{2b\delta_1^{-1}...\delta_{n+1}^{-1}}\right)^{\delta_{n+1}}+\left(\mathbb{E}\|Z^N_{n+1}\|^{2b\delta_1^{-1}...\delta_{n+1}^{-1}}\right)^{\delta_{n+1}}\right]\notag\\
&\qquad\quad+e^{-\lambda (n+1)^\gamma}\left(\mathbb{E}\|\bar{f}^N(n+1,0,0)\|^{2b\delta_1^{-1}...\delta_{n+1}^{-1}}\right)^{\delta_{n+1}}\notag\\
&\qquad\quad+e^{-\lambda (n+1)^\gamma}\left(\mathbb{E}\|\bar{g}^N(n+1,0,0)\|^{2b\delta_1^{-1}...\delta_{n+1}^{-1}}\right)^{\delta_{n+1}}\Bigg\}\notag\\
\le& Ce^{-\lambda(2b\delta_1^{-1}...\delta_n^{-1}-1) \left((n+1)^\gamma-n^{\gamma}\right)}\cdot (1-\delta_{n+1})^{-\frac{1}{2}}\notag\\
&\quad \times \Bigg\{e^{-\lambda (n+1)^\gamma}\left[\mathbb{E}\|Y^N_{n+1}\|^{2b\delta_1^{-1}...\delta_{n+1}^{-1}}+\mathbb{E}\|Z^N_{n+1}\|^{2b\delta_1^{-1}...\delta_{n+1}^{-1}}\right]\notag\\
&\qquad\quad+e^{-\lambda (n+1)^\gamma}\mathbb{E}\|\bar{f}^N(n+1,0,0)\|^{2b\delta_1^{-1}...\delta_{n+1}^{-1}}\notag\\
&\qquad\quad+e^{-\lambda (n+1)^\gamma}\mathbb{E}\|\bar{g}^N(n+1,0,0)\|^{2b\delta_1^{-1}...\delta_{n+1}^{-1}}+4e^{-\lambda (n+1)^\gamma}\Bigg\},
\end{align}
where $C$ is a constant independent with $N,M,n,\lambda,\gamma$, which may vary from line to line. Notice that 
\begin{align*}
\left(\mathbb{E}[\eta_n\xi_{n+1}^H]^m\right)^{\frac{1}{m}}\le C\left(\left(\mathbb{E}[\eta_n]^{2m}\right)^{\frac{1}{m}}+\left(\mathbb{E}[\xi_{n+1}^H]^{2m}\right)^{\frac{1}{m}}\right)\le Cm, \quad\forall m>1.
\end{align*}
Then similar to equation (\ref{2.10}), we derive
\begin{align}\label{2.11}
&e^{-\lambda n^{\gamma}}\left(\mathbb{E}|Z_n^N\|^{2b\delta_1^{-1}...\delta_{n}^{-1}}\right)\notag\\\le& Ce^{-\lambda(2b\delta_1^{-1}...\delta_n^{-1}-1) \left((n+1)^\gamma-n^{\gamma}\right)}\cdot \left((1-\delta_{n+1})^{-\frac{1}{2}}+(1-\delta_{n+1})^{-1}\right)\notag\\
&\quad \times \Bigg\{e^{-\lambda (n+1)^\gamma}\left[\mathbb{E}\|Y^N_{n+1}\|^{2b\delta_1^{-1}...\delta_{n+1}^{-1}}+\mathbb{E}\|Z^N_{n+1}\|^{2b\delta_1^{-1}...\delta_{n+1}^{-1}}\right]\notag\\
&\qquad\quad+e^{-\lambda (n+1)^\gamma}\mathbb{E}\|\bar{f}^N(n+1,0,0)\|^{2b\delta_1^{-1}...\delta_{n+1}^{-1}}\notag\\
&\qquad\quad+e^{-\lambda (n+1)^\gamma}\mathbb{E}\|\bar{g}^N(n+1,0,0)\|^{2b\delta_1^{-1}...\delta_{n+1}^{-1}}+4e^{-\lambda (n+1)^\gamma}\Bigg\}.
\end{align}
Combining equations (\ref{2.10}) and (\ref{2.11}), we conclude
\begin{align}\label{2.12}
&e^{-\lambda n^{\gamma}}\left[\mathbb{E}|Y_n^N\|^{2b\delta_1^{-1}...\delta_{n}^{-1}}+\mathbb{E}|Z_n^N\|^{2b\delta_1^{-1}...\delta_{n}^{-1}}\right]\notag\\
\le&Ce^{-\lambda(2b\delta_1^{-1}...\delta_n^{-1}-1) \left((n+1)^\gamma-n^{\gamma}\right)}\cdot \left(1+(1-\delta_n)^{-\frac{1}{2}}+(1-\delta_n)^{-1}\right)\notag\\
&\quad \times \Bigg\{e^{-\lambda (n+1)^\gamma}\left[\mathbb{E}\|Y^N_{n+1}\|^{2b\delta_1^{-1}...\delta_{n+1}^{-1}}+\mathbb{E}\|Z^N_{n+1}\|^{2b\delta_1^{-1}...\delta_{n+1}^{-1}}\right]\notag\\
&\qquad\quad+e^{-\lambda (n+1)^\gamma}\mathbb{E}\|\bar{f}^N(n+1,0,0)\|^{2b\delta_1^{-1}...\delta_{n+1}^{-1}}\notag\\
&\qquad\quad+e^{-\lambda (n+1)^\gamma}\mathbb{E}\|\bar{g}^N(n+1,0,0)\|^{2b\delta_1^{-1}...\delta_{n+1}^{-1}}+4e^{-\lambda (n+1)^\gamma}\Bigg\}\notag\\
\le&Ce^{-\lambda(2b-1) \left((n+1)^\gamma-n^{\gamma}\right)}\cdot (n+3)^\theta\notag\\
&\quad \times \Bigg\{e^{-\lambda (n+1)^\gamma}\left[\mathbb{E}\|Y^N_{n+1}\|^{2b\delta_1^{-1}...\delta_{n+1}^{-1}}+\mathbb{E}\|Z^N_{n+1}\|^{2b\delta_1^{-1}...\delta_{n+1}^{-1}}\right]\notag\\
&\qquad\quad+e^{-\lambda (n+1)^\gamma}\mathbb{E}\|\bar{f}^N(n+1,0,0)\|^{2b\delta_1^{-1}...\delta_{n+1}^{-1}}\notag\\
&\qquad\quad+e^{-\lambda (n+1)^\gamma}\mathbb{E}\|\bar{g}^N(n+1,0,0)\|^{2b\delta_1^{-1}...\delta_{n+1}^{-1}}+4e^{-\lambda (n+1)^\gamma}\Bigg\}.
\end{align}
Choose a $\tilde{N_0}>1,$ such that
\begin{align}\label{2.13}
Ce^{-\lambda(2b-1) \left((n+1)^\gamma-n^{\gamma}\right)}\cdot (n+3)^\theta\le \frac{1}{2},\quad \forall n>\tilde{N_0}.
\end{align}
Then it is easy to get
\begin{align}\label{2.14}
&e^{-\lambda M^{\gamma}}\left[\mathbb{E}\|Y_M^N\|^{2b\delta_1^{-1}...\delta_{M}^{-1}}+\mathbb{E}\|Z_M^N\|^{2b\delta_1^{-1}...\delta_{M}^{-1}}\right]\notag\\=&\sum_{n=M}^{N-1}e^{-\lambda n^{\gamma}}\left[\mathbb{E}\|Y_n^N\|^{2b\delta_1^{-1}...\delta_{n}^{-1}}+\mathbb{E}\|Z_n^N\|^{2b\delta_1^{-1}...\delta_{n}^{-1}}\right]\notag\\&\qquad-e^{-\lambda (n+1)^{\gamma}}\left[\mathbb{E}\|Y_{n+1}^N\|^{2b\delta_1^{-1}...\delta_{n+1}^{-1}}+\mathbb{E}\|Z_{n+1}^N\|^{2b\delta_1^{-1}...\delta_{n+1}^{-1}}\right]\notag\\
\le&\sum_{n=M}^{N-1}e^{-\lambda (n+1)^{\gamma}}\notag\\
&\qquad\times\left[\mathbb{E}\|\bar{f}^N(n+1,0,0)\|^{2b\delta_1^{-1}...\delta_{n+1}^{-1}}+\mathbb{E}\|\bar{g}^N(n+1,0,0)\|^{2b\delta_1^{-1}...\delta_{n+1}^{-1}}+4\right]\notag\\
\to &0,
\end{align}
as $M\to+\infty.$

Divide $\|Y^{M,N}_\cdot\|_{\hat{L}^{2b,\vec{\delta}^\theta,\lambda,\gamma}_\mathcal{F}}+\|Z^{M,N}_\cdot\|_{\hat{L}^{2b,\vec{\delta}^\theta,\lambda,\gamma}_\mathcal{F}}$ into three parts:
\begin{align*}
&\|Y^{M,N}_\cdot\|_{\hat{L}^{2b,\vec{\delta}^\theta,\lambda,\gamma}_\mathcal{F}}+\|Z^{M,N}_\cdot\|_{\hat{L}^{2b,\vec{\delta}^\theta,\lambda,\gamma}_\mathcal{F}}\\=&\sum_{n=0}^{N_0-1}e^{-\lambda n^\gamma}\mathbb{E}\left(\|Y^{M,N}_n\|^{2b \delta_1^{-1}...\delta_n^{-1}}+\|Z^{M,N}_n\|^{2b \delta_1^{-1}...\delta_n^{-1}}\right)\\
&+\sum_{n=N_0}^{M-1}e^{-\lambda n^\gamma}\mathbb{E}\left(\|Y^{M,N}_n\|^{2b \delta_1^{-1}...\delta_n^{-1}}+\|Z^{M,N}_n\|^{2b \delta_1^{-1}...\delta_n^{-1}}\right)\\
&+\sum_{n=M}^{N-1}e^{-\lambda n^\gamma}\mathbb{E}\left(\|Y^{M,N}_n\|^{2b \delta_1^{-1}...\delta_n^{-1}}+\|Z^{M,N}_n\|^{2b \delta_1^{-1}...\delta_n^{-1}}\right)\\
:=&I_1+I_2+I_3.
\end{align*}
For $I_3$, taking summation of inequality (\ref{2.12}) and through inequality (\ref{2.13}), we have
\begin{align}\label{2.15}
I_3\le&\sum_{n=M+1}^{N}e^{-\lambda n^{\gamma}}\left[\mathbb{E}\|\bar{f}^N(n,0,0)\|^{2b\delta_1^{-1}...\delta_{n}^{-1}}+\mathbb{E}\|\bar{g}^N(n,0,0)\|^{2b\delta_1^{-1}...\delta_{n}^{-1}}+4\right]\notag\\
\to &0.
\end{align}
For $n<M$, similar to the case $n\in[M,N-1]$, we have
\begin{align}\label{2.16}
&e^{-\lambda n^{\gamma}}\left[\mathbb{E}\|Y_n^{M,N}\|^{2b\delta_1^{-1}...\delta_{n}^{-1}}+\mathbb{E}\|Z_n^{M,N}\|^{2b\delta_1^{-1}...\delta_{n}^{-1}}\right]\notag\\
\le&Ce^{-\lambda(2b-1) \left((n+1)^\gamma-n^{\gamma}\right)}\cdot (n+3)^\theta\notag\\
&\quad \times \left\{e^{-\lambda (n+1)^\gamma}\left[\mathbb{E}\|Y^{M,N}_{n+1}\|^{2b\delta_1^{-1}...\delta_{n+1}^{-1}}+\mathbb{E}\|Z^{M,N}_{n+1}\|^{2b\delta_1^{-1}...\delta_{n+1}^{-1}}\right]\right\}^{\delta_{n+1}}.
\end{align}
According to inequality (\ref{2.14}), let $M$ is big enough, such that
\begin{align*}
e^{-\lambda M^\gamma}\left[\mathbb{E}\|Y^{M,N}_{M}\|^{2b\delta_1^{-1}...\delta_{M}^{-1}}+\mathbb{E}\|Z^{M,N}_{M}\|^{2b\delta_1^{-1}...\delta_{M}^{-1}}\right]\le 1.
\end{align*}
Then for $n\in[\tilde{N_0},M-1]$, through induction, we have
\begin{align*}
\sup_{n\in[\tilde{N_0},M-1]}e^{-\lambda n^\gamma}\left[\mathbb{E}\|Y^{M,N}_{n}\|^{2b\delta_1^{-1}...\delta_{n}^{-1}}+\mathbb{E}\|Z^{M,N}_{n}\|^{2b\delta_1^{-1}...\delta_{n}^{-1}}\right]\le 1.
\end{align*}
So that, through inequality (\ref{2.16}), we obtain
\begin{align*}
&e^{-\lambda n^{\gamma}}\left[\mathbb{E}\|Y_n^{M,N}\|^{2b\delta_1^{-1}...\delta_{n}^{-1}}+\mathbb{E}\|Z_n^{M,N}\|^{2b\delta_1^{-1}...\delta_{n}^{-1}}\right]\notag\\
\le&\frac{1}{2}e^{-\lambda (n+1)^\gamma}\left[\mathbb{E}\|Y^{M,N}_{n+1}\|^{2b\delta_1^{-1}...\delta_{n+1}^{-1}}+\mathbb{E}\|Z^{M,N}_{n+1}\|^{2b\delta_1^{-1}...\delta_{n+1}^{-1}}\right]\notag\\
\le &\frac{1}{2^{M-n}}e^{-\lambda M^\gamma}\left[\mathbb{E}\|Y^{M,N}_{M}\|^{2b\delta_1^{-1}...\delta_{M}^{-1}}+\mathbb{E}\|Z^{M,N}_{M}\|^{2b\delta_1^{-1}...\delta_{M}^{-1}}\right],
\end{align*}
for $n\in[\tilde{N_0},M-1]$. It follows
\begin{align}\label{2.17}
I_2\le&\sum_{n=\tilde{N_0}}^{M-1}\frac{1}{2^{M-n}}e^{-\lambda M^\gamma}\left[\mathbb{E}\|Y^{M,N}_{M}\|^{2b\delta_1^{-1}...\delta_{M}^{-1}}+\mathbb{E}\|Z^{M,N}_{M}\|^{2b\delta_1^{-1}...\delta_{M}^{-1}}\right]\notag\\
\le&e^{-\lambda M^\gamma}\left[\mathbb{E}\|Y^{M,N}_{M}\|^{2b\delta_1^{-1}...\delta_{M}^{-1}}+\mathbb{E}\|Z^{M,N}_{M}\|^{2b\delta_1^{-1}...\delta_{M}^{-1}}\right]\notag\\
\to& 0,\quad M\to+\infty.
\end{align}

Then we deal with $I_1$. Define
\begin{align*}
C_1=1+\sup_{n\in[0,\tilde{N_0}-1]}Ce^{-\lambda(2b-1) \left((n+1)^\gamma-n^{\gamma}\right)}\cdot (n+3)^\theta.
\end{align*}
Then trough inequality (\ref{2.16}), we have
\begin{align*}
&e^{-\lambda n^{\gamma}}\left[\mathbb{E}\|Y_n^{M,N}\|^{2b\delta_1^{-1}...\delta_{n}^{-1}}+\mathbb{E}\|Z_n^{M,N}\|^{2b\delta_1^{-1}...\delta_{n}^{-1}}\right]\\\le& C_1^{\tilde{N_0}-n}e^{-\lambda \tilde{N_0}^\gamma}\left[\mathbb{E}\|Y^{M,N}_{\tilde{N_0}}\|^{2b\delta_1^{-1}...\delta_{\tilde{N_0}}^{-1}}+\mathbb{E}\|Z^{M,N}_{\tilde{N_0}}\|^{2b\delta_1^{-1}...\delta_{\tilde{N_0}}^{-1}}\right].
\end{align*}
Taking summation from $0$ to $\tilde{N_0}-1$, we have
\begin{align}\label{2.18}
I_1\le&\sum_{n=0}^{\tilde{N_0}-1}C_1^{\tilde{N_0}-n}e^{-\lambda \tilde{N_0}^\gamma}\left[\mathbb{E}\|Y^{M,N}_{\tilde{N_0}}\|^{2b\delta_1^{-1}...\delta_{\tilde{N_0}}^{-1}}+\mathbb{E}\|Z^{M,N}_{\tilde{N_0}}\|^{2b\delta_1^{-1}...\delta_{\tilde{N_0}}^{-1}}\right]\notag\\
=&\frac{C_1^{\tilde{N_0}}-1}{C_1-1}e^{-\lambda \tilde{N_0}^\gamma}\left[\mathbb{E}\|Y^{M,N}_{\tilde{N_0}}\|^{2b\delta_1^{-1}...\delta_{\tilde{N_0}}^{-1}}+\mathbb{E}\|Z^{M,N}_{\tilde{N_0}}\|^{2b\delta_1^{-1}...\delta_{\tilde{N_0}}^{-1}}\right]\notag\\
\to&0, \quad M\to+\infty,
\end{align}
since inequality (\ref{2.17}) implies 
\begin{align*}
e^{-\lambda \tilde{N_0}^\gamma}\left[\mathbb{E}\|Y^{M,N}_{\tilde{N_0}}\|^{2b\delta_1^{-1}...\delta_{\tilde{N_0}}^{-1}}+\mathbb{E}\|Z^{M,N}_{\tilde{N_0}}\|^{2b\delta_1^{-1}...\delta_{\tilde{N_0}}^{-1}}\right]\to 0,\, M\to+\infty.
\end{align*}
Combining inequality (\ref{2.15}), (\ref{2.17}) and (\ref{2.18}), we conclude
\begin{align*}
\|Y^{M,N}_\cdot\|_{\hat{L}^{2b,\vec{\delta}^\theta,\lambda,\gamma}_\mathcal{F}}+\|Z^{M,N}_\cdot\|_{\hat{L}^{2b,\vec{\delta}^\theta,\lambda,\gamma}_\mathcal{F}}\to 0,\quad M,N\to+\infty.
\end{align*}
So that $(Y^N,Z^N)$ is a Cauchy sequence, and define
\begin{align*}
Y_n=\lim_{N\to+\infty}Y_n^N,\quad Z_n=\lim_{N\to+\infty}Z_n^N.
\end{align*}

\textbf{Uniqueness.} Assume that $(Y,Z)$ and $(Y',Z')$ are two solutions of BS$\Delta$Es (\ref{BSDE}). Denote $(\tilde{Y},\tilde{Z})=(Y-Y',Z-Z')$. Similar to the proof of existence, we have
\begin{align*}
&e^{-\lambda n^{\gamma}}\left[\mathbb{E}\|\tilde{Y}_n\|^{2b\delta_1^{-1}...\delta_{n}^{-1}}+\mathbb{E}\|\tilde{Z}_n\|^{2b\delta_1^{-1}...\delta_{n}^{-1}}\right]\\\le &\frac{1}{2}e^{-\lambda N^{\gamma}}\left[\mathbb{E}\|\tilde{Y}_N\|^{2b\delta_1^{-1}...\delta_{N}^{-1}}+\mathbb{E}\|\tilde{Z}_N\|^{2b\delta_1^{-1}...\delta_{N}^{-1}}\right],\quad  n\in[\tilde{N_0},N-1],
\end{align*}
and
\begin{align*}
&e^{-\lambda n^{\gamma}}\left[\mathbb{E}\|\tilde{Y}_n\|^{2b\delta_1^{-1}...\delta_{n}^{-1}}+\mathbb{E}\|\tilde{Z}_n\|^{2b\delta_1^{-1}...\delta_{n}^{-1}}\right]\\\le &C_1^{\tilde{N_0}}e^{-\lambda \tilde{N_0}^{\gamma}}\left[\mathbb{E}\|\tilde{Y}_{\tilde{N_0}}\|^{2b\delta_1^{-1}...\delta_{\tilde{N_0}}^{-1}}+\mathbb{E}\|\tilde{Z}_{\tilde{N_0}}\|^{2b\delta_1^{-1}...\delta_{\tilde{N_0}}^{-1}}\right],\quad  n\in[0,\tilde{N_0}-1].
\end{align*}
Let $N\to +\infty$, the above two inequality mean
\begin{align*}
\sup_{n\ge0}e^{-\lambda n^{\gamma}}\left[\mathbb{E}\|\tilde{Y}_n\|^{2b\delta_1^{-1}...\delta_{n}^{-1}}+\mathbb{E}\|\tilde{Z}_n\|^{2b\delta_1^{-1}...\delta_{n}^{-1}}\right]=0.
\end{align*}
The proof of Theorem \ref{BSDE e,u} completes.
\end{proof}

\section{A stochastic maximum principle}
In this section, we consider the following optimal control problem. The state equation is
\begin{align}\label{state}
\left\{\begin{array}{ll}
X_{n+1}=X_n+b(n,X_n,u_n)+\sigma(n,X_n,u_n)\xi_n^H,\, n\in\mathcal{T},
\\X_0=x,
\end{array}\right.
\end{align}
with the cost functional
\begin{align}\label{cost}
J(u)=Y_0,
\end{align}
$(Y,Z)$ is the solution to the following BS$\Delta$Es:
\begin{align}\label{costBSDE}
\left\{\begin{array}{ll}
e^{-\lambda n^\gamma}\left(Y_n+Z_n\eta_n\right)=e^{-\lambda (n+1)^\gamma}\left(Y_{n+1}+f(n+1,X_{n+1},Y_{n+1},Z_{n+1},u_{n+1})\right),
\\\\\lim_{N\to+\infty}\mathbb{E}[e^{-\lambda N^{\gamma}}\|Y_N\|^{2b\delta_1^{-1}...\delta_{N}^{-1}}]=0.
\end{array}\right.
\end{align}
Here $x\in L^{2a}(\mathcal{F}_0)$ is independent with $\{\xi_n\}$, 
$b(n,x,u)$, $\sigma(n,x,u) $  are differential measurable functions on $\mathcal{T}\times \mathbb{R}^d\times \mathbb{R}^k$ with values in $\mathbb{R}^d$, and $f(n,x,y,z,u)$ is differential measurable functions on $\mathcal{T}\times \mathbb{R}^d\times\mathbb{R}\times\mathbb{R}\times \mathbb{R}^k$ with values in $\mathbb{R}$. Moreover, the cost functional (\ref{cost}) is equivalent to
\begin{align*}
J(u)=\mathbb{E}\sum_{n=1}^{+\infty}e^{-\lambda n^{\gamma}}f(n,X_n,Y_n,Z_n,u_n).
\end{align*}

Denote by $\mathbb{U}$ the set of progressively measurable process
\textbf{u}$=(u_n)_{0\le n\le N-1}$ taking values in a given closed-convex set $\textbf{U}\subset \mathbb{R}^k$ and satisfying $$\sum_{n=0}^{+\infty}e^{-\lambda n^\gamma}\mathbb{E}\left(\|u_n\|^{2a \delta_1...\delta_n}\right)<+\infty.$$ The problem is to find an optimal control $u^*\in\mathbb{U}$ to minimized the cost functional, i.e.,
\begin{align*}
J(u^*)=\inf_{u\in\mathbb{U}}J(u).
\end{align*}

To simplify the notation without losing the generality, we assume that $d=k=1$. We give the following assumptions:

\begin{enumerate}
	\item[(H3.1)]  $b$, $\sigma$ and $f$ are adapted processes: $b(\cdot,x,u), \sigma(\cdot,x,u)\in L^{2a,\vec{\delta}^\theta,\lambda,\gamma}_{\mathcal{F}}(0,+\infty)$, $f(\cdot,x,y,z,u)\in \hat{L}^{2b,\vec{\delta}^\theta,\lambda,\gamma}_{\mathcal{F}}(0,+\infty)$\, $\forall x,u\in\mathbb{R}$, for some $\lambda>0, \theta, \gamma>1$, such that $a|\vec{\delta}^\theta|^2\ge b\ge1$.

\item[(H3.2)] There exists some constants $L>0$, such that 
\begin{align*}
\|b(n,x_1,u_1)-b(n,x_2,u_2)\|&+\|\sigma(n,x_1,u_1)-\sigma(n,x_2,u_2)\|\le L\left(\|x_1-x_2\|+\|u_1-u_2\|\right),\\
&\forall n\in\mathcal{T},\quad x_1,x_2,u_1,u_2\in\mathbb{R},
\end{align*}
and
\begin{align*}
&\|f_x(n,x_1,y_1,z_1,u_1)-f_x(n,x_2,y_2,z_2,u_2)\|\\ 
\le& L\left(\|x_1-x_2\|+\|y_1-y_2\|+\|z_1-z_2\|+\|u_1-u_2\|\right),
\end{align*}
for all $n\in\mathcal{T},\quad x_i,y_i,z_i,u_i\in\mathbb{R}, i=1,2.$

\item[(H3.3)] There exists a random function $f_1(n,x,y,u)$, such that $f(\cdot,x,y,u)\in\hat{L}^{2b,\vec{\delta}^\theta,\lambda,\gamma}_{\mathcal{F}}(0,+\infty)$ and
\begin{align*}
\lim_{N\to+\infty}\mathbb{E}\left[e^{-\lambda N^{\gamma}}\|f(N,x,y,z,u)-f(N,x,y,u)\|^{2b\delta_1^{-1}...\delta_{N}^{-1}}\right]=0,
\end{align*}
for all $x,y,u,z\in\mathbb{R}$. Moreover, $f$ is differentialble and uniformly Lipschitz continuous w.r.t $(x,y,u)$ for all $n\in\mathcal{T}$.
\end{enumerate}

\begin{remark}
For a given $u\in\mathbb{U}$, under the assumptions (H3.1)-(H3.3) and combining with Theorem \ref{theSDE} and Theorem \ref{BSDE e,u}, it is easy to prove the existence and uniqueness of the solutions to S$\Delta$Es (\ref{state}) and BS$\Delta$Es (\ref{costBSDE}).
\end{remark}

Let $u^*$ be the optimal control and $X^*, Y^*,Z^*$ be the corresponding processes. For all $\varepsilon\in(0,1)$ and $\tilde{u}\in\mathbb{U}$, let
\begin{align*}
u^\varepsilon_n=(1-\varepsilon) u^*_n+\varepsilon\tilde{u}_n:=u^*_n+\varepsilon v_n,
\end{align*}
and $X^\varepsilon,Y^\varepsilon,Z^\varepsilon$ be the corresponding processes.
For notations simplify, set
\begin{align*}
l^*(n)=l(n,X^*_n,u^*_n),\quad h^*(n)=h(n,X^*_n,Y^*_n,Z^*_n,u^*_n),
\end{align*}
for $l=b,\sigma,b_x,\sigma_x,b_u,\sigma_u$ and $h=f,f_x,f_y,f_z,f_u$.

Define the variation equations by
\begin{align}\label{varSDE}
\left\{\begin{array}{ll}
\hat{X}_{n+1}=\hat{X}_n+b_x^*(n)\hat{X}_n+b_u^*(n)v_n+\left(\sigma_x^*(n)\hat{X}_n+\sigma_u^*(n)v_n\right)\xi_n^H,\, n\in\mathcal{T},
\\\hat{X}_0=0,
\end{array}\right.
\end{align}
and 
\begin{align}\label{varBSDE}
\left\{\begin{array}{ll}
e^{-\lambda n^\gamma}\left(\hat{Y}_n+\hat{Z}_n\eta_n\right)=e^{-\lambda (n+1)^\gamma}\left(\hat{Y}_{n+1}+f_x^*(n+1)\hat{X}_{n+1}+f_y^*(n+1)\hat{Y}_{n+1}\right),\\\\
\qquad\qquad\qquad\qquad\qquad+e^{-\lambda (n+1)^\gamma}\left(f_z^*(n+1)\hat{Z}_{n+1}+f_u^*(n+1)v_{n+1}\right)
\\\\\lim_{N\to+\infty}\mathbb{E}\left[e^{-\lambda N^{\gamma}}\hat{Y}_N^{2bp_1^{-1}...p_{N}^{-1}}\right]=0.
\end{array}\right.
\end{align}
Then we have following convergence results:
\begin{lemma}
Let $(\hat{X},\hat{Y},\hat{Z})$ defined as equations (\ref{varSDE}) and (\ref{varBSDE}), $(X^*,Y^*,Z^*)$ and $(X^\varepsilon,Y^\varepsilon,Z^\varepsilon)$ be the processes corresponding to $u^*$ and $u^\varepsilon$, respectively. Then we have 
\begin{align}
&\sum_{n=0}^{+\infty}e^{-\lambda n^\gamma}\mathbb{E}|X^\varepsilon_n-X^*_n|^2\le O(\varepsilon^2),\label{3.6}\\
&\sum_{n=0}^{+\infty}e^{-\lambda n^\gamma}\mathbb{E}\left(|Y^\varepsilon_n-Y^*_n|^2+|Z^\varepsilon_n-Z^*_n|^2\right)\le O(\varepsilon^2),\label{3.7}\\
&\lim_{\varepsilon\to0}\sum_{n=0}^{+\infty}e^{-\lambda n^\gamma}\mathbb{E}\left|\frac{X^\varepsilon_n-X^*_n}{\varepsilon}-\hat{X}_n\right|^2=0,\label{3.8}\\
&\lim_{\varepsilon\to0}\sum_{n=0}^{+\infty}e^{-\lambda n^\gamma}\mathbb{E}\left(\left|\frac{Y^\varepsilon_n-Y^*_n}{\varepsilon}-\hat{Y}_n\right|^2+\left|\frac{Z^\varepsilon_n-Z^*_n}{\varepsilon}-\hat{Z}_n\right|^2\right)=0.\label{3.9}
\end{align}
\end{lemma}
\begin{proof}
We only proof equations (\ref{3.8}), (\ref{3.9}), which are more complex. For equation (\ref{3.8}), denote $\tilde{X}^\varepsilon_n=\frac{X^\varepsilon_n-X^*_n}{\varepsilon}-\hat{X}_n$, and 
\begin{align*}
\tilde{\phi}^\varepsilon(n)=\int_0^1 \phi(n,X^*_n+\theta(X^\varepsilon_n-X^*_n),u^*_n+\theta\varepsilon v_n)d\theta,
\end{align*}
for $\phi=b_x,b_u,\sigma_x,\sigma_u,f_x,f_y,f_z,f_u$. So that
\begin{align*}
\tilde{X}^\varepsilon_{n+1}=&\tilde{X}^\varepsilon_{n}+\frac{b^\varepsilon(n)-b^*(n)}{\varepsilon}-b_x^*(n)V_n-b_u^*(n)v_n\\
&+\left[\frac{\sigma^\varepsilon(n)-\sigma^*(n)}{\varepsilon}-\sigma_x^*(n)V_n-\sigma_u^*(n)v_n\right]\xi_n^H\\
=&[1+b_x^*(n)]\tilde{X}^\varepsilon_n+\frac{\tilde{b}_x^\varepsilon(n)-b^*_x(n)}{\varepsilon}[X^\varepsilon_n-X^*_n]+[\tilde{b}_u^\varepsilon(n)-b^*_u(n)]v_n\\
&+\left[\sigma_x^*(n)\tilde{X}^\varepsilon_n+\frac{\tilde{\sigma}_x^\varepsilon(n)-\sigma^*_x(n)}{\varepsilon}[X^\varepsilon_n-X^*_n]+[\tilde{\sigma}_u^\varepsilon(n)-\sigma^*_u(n)]v_n\right]\xi^H_n.
\end{align*}
Then
\begin{align*}
&\mathbb{E}\left|\tilde{X}^\varepsilon_{n+1}\right|^{2a\delta_1...\delta_{n+1}}\\
\le& C\mathbb{E}\left|\tilde{X}^\varepsilon_{n}\right|^{2a\delta_1...\delta_{n+1}}+C\mathbb{E}\|\tilde{b}_u^\varepsilon(n)-b^*_u(n)\|^{2a\delta_1...\delta_{n+1}}|v_n|^{2a\delta_1...\delta_{n+1}}\\
&+C\mathbb{E}\|\tilde{b}_x^\varepsilon(n)-b^*_x(n)\|^{2a\delta_1...\delta_{n+1}}\left|\varepsilon^{-1}(X^\varepsilon_n-X^*_n)\right|^{2a\delta_1...\delta_{n+1}}\\
&+C\mathbb{E}\Bigg[\left|\tilde{X}^\varepsilon_{n}\right|^{2a\delta_1...\delta_{n+1}}+\|\tilde{\sigma}_u^\varepsilon(n)-\sigma^*_u(n)\|^{2a\delta_1...\delta_{n+1}}|v_n|^{2a\delta_1...\delta_{n+1}}\\
&\qquad+\|\tilde{\sigma}_x^\varepsilon(n)-\sigma^*_x(n)\|^{2a\delta_1...\delta_{n+1}}\left|\varepsilon^{-1}(X^\varepsilon_n-X^*_n)\right|^{2a\delta_1...\delta_{n+1}}\Bigg]\left|\xi_n^H\right|^{2a\delta_1...\delta_{n+1}}\\
\le&C\left[\left(\mathbb{E}\left|\tilde{X}^\varepsilon_{n}\right|^{2a\delta_1...\delta_n}\right)^{\delta_{n+1}}+o(1)\right]\times\left[1+\left(\mathbb{E}|\xi_n^H|^{\frac{2a\delta_1...\delta_{n+1}}{1-\delta_{n+1}}}\right)^{1-\delta_{n+1}}\right]\\
\le &C(n+3)^\theta\left[\left(\mathbb{E}\left|\tilde{X}^\varepsilon_{n}\right|^{2a\delta_1...p_n}\right)^{\delta_{n+1}}+o(1)\right],\quad \varepsilon\to 0.
\end{align*}
Choose a fixed constant $N_0>1$, such that 
\begin{align*}
\mathbb{E}\left|e^{-\frac{\lambda}{2}(n+1)^\gamma}\tilde{X}^\varepsilon_{n+1}\right|^{2a\delta_1...\delta_{n+1}}\le\frac{1}{2}\left[\left(\mathbb{E}\left|e^{-\frac{\lambda}{2}n^\gamma}\tilde{X}^\varepsilon_{n}\right|^{2a\delta_1...\delta_n}\right)^{\delta_{n+1}}+o(1)\right],
\end{align*}
for $n\ge N_0$. Combining with $\mathbb{E}|\tilde{X}^\varepsilon_0|^{2a}=0$, we have
\begin{align}\label{3.10}
e^{-\lambda n^\gamma}\mathbb{E}|\tilde{X}^\varepsilon_n|^{2}\le\left(\mathbb{E}|e^{-\frac{\lambda}{2}n^\gamma}\tilde{X}^\varepsilon_n|^{2a\delta_1...\delta_n}\right)^{a^{-1}\delta_1^{-1}...\delta_n^{-1}}\le o(1), \quad n\in[1,N_0],
\end{align}
and 
\begin{align*}
&\mathbb{E}\left|e^{-\frac{\lambda}{2}(n+1)^\gamma}\tilde{X}^\varepsilon_{n+1}\right|^{2a\delta_1...\delta_{n+1}}\\\le&\frac{1}{2}\left[\left(\mathbb{E}\left|e^{-\frac{\lambda}{2}n^\gamma}\tilde{X}^\varepsilon_{n}\right|^{2a\delta_1...\delta_n}\right)^{\delta_{n+1}}+o(1)\right]\\
\le&\frac{1}{2^{n+1-N_0}}\left[\left(\mathbb{E}\left|e^{-\frac{\lambda}{2}N_0^\gamma}\tilde{X}^\varepsilon_{N_0}\right|^{\delta_1...\delta_{N_0}}\right)^{\delta_{N_0+1}...\delta_{n+1}}+o(1)\right],\quad n\ge N_0,
\end{align*}
which means
\begin{align}\label{3.11}
e^{-\lambda n^\gamma}\mathbb{E}|\tilde{X}^\varepsilon_n|^{2}&\le\left(\mathbb{E}|e^{-\frac{\lambda}{2}n^\gamma}\tilde{X}^\varepsilon_n|^{2a\delta_1...\delta_n}\right)^{a^{-1}\delta_1^{-1}...\delta_n^{-1}}\notag\\
&\le\frac{1}{2^{a(n-N_0)}}\left[\mathbb{E}\left|e^{-\frac{\lambda}{2}N_0^\gamma}\tilde{X}^\varepsilon_{N_0}\right|^2+o(1)\right], \quad n\ge N_0+1, \, \varepsilon\to 0.
\end{align}
Through equations (\ref{3.10}) and (\ref{3.11}), we obtain equation (\ref{3.8}).

The proof of equation (\ref{3.9}) is easier, since none of the equations of $(\hat{Y},\hat{Z}), (Y^*,Z^*),(Y^\varepsilon,Z^\varepsilon)$ contain fractional noise. Denote $\tilde{Y}^\varepsilon_n=\frac{Y^\varepsilon_n-Y^*_n}{\varepsilon}-\hat{Y}_n$ and $\tilde{Z}^\varepsilon_n=\frac{Z^\varepsilon_n-Z^*_n}{\varepsilon}-\hat{Z}_n$. Then
\begin{align*}
&e^{-\lambda n^\gamma}(\tilde{Y}^\varepsilon_n+\tilde{Z}_n^\varepsilon\eta_n)\\
=&e^{-\lambda(n+1)^\gamma}\Big[\tilde{Y}^\varepsilon_{n+1}+f_x^*(n+1)\tilde{X}^\varepsilon_{n+1}+f_y^*(n+1)\tilde{Y}^\varepsilon_{n+1}+f_z^*(n+1)\tilde{Z}^\varepsilon_{n+1}\\
&\quad\qquad\qquad+\frac{\tilde{f}_x^\varepsilon(n+1)-f^*_x(n+1)}{\varepsilon}[X^\varepsilon_{n+1}-X^*_{n+1}]\\
&\quad\qquad\qquad+\frac{\tilde{f}_y^\varepsilon(n+1)-f^*_y(n+1)}{\varepsilon}[Y^\varepsilon_{n+1}-Y^*_{n+1}]\\
&\quad\qquad\qquad+\frac{\tilde{f}_z^\varepsilon(n+1)-f^*_z(n+1)}{\varepsilon}[Z^\varepsilon_{n+1}-Z^*_{n+1}]+[\tilde{f}_u^\varepsilon(n+1)-f^*_u(n+1)]v_{n+1}
\Big].
\end{align*}
Taking $e^{\lambda n^\gamma}\mathbb{E}[|\cdot|^2]$ on both sides and combining with equation (\ref{3.8}), choose a fixed $\tilde{N_0}>1$, such that for $n\ge \tilde{N_0}$,
\begin{align*}
&e^{-\lambda n^\gamma}\mathbb{E}\left(|\tilde{Y}_n^\varepsilon|^2+|\tilde{Z}_n^\varepsilon|^2\right)\\
\le&\frac{1}{2}e^{-\lambda (n+1)^\gamma}\mathbb{E}\left(|\tilde{Y}_{n+1}^\varepsilon|^2+|\tilde{Z}_{n+1}^\varepsilon|^2\right)\\
&+o(1)\frac{1}{\varepsilon^2}\times e^{-\lambda (n+1)^\gamma}\mathbb{E}\Big(
|X^\varepsilon_{n+1}-X^*_{n+1}|^2+|Y^\varepsilon_{n+1}-Y^*_{n+1}|^2+|Z^\varepsilon_{n+1}-Z^*_{n+1}|^2\Big).
\end{align*}
So we have
\begin{align}\label{3.12}
&\sum_{n=\tilde{N_0}}^{+\infty}e^{-\lambda n^\gamma}\mathbb{E}\left(|\tilde{Y}_n^\varepsilon|^2+|\tilde{Z}_n^\varepsilon|^2\right)\notag\\
\le &o(1)\times\frac{1}{\varepsilon^2}\sum_{n={\tilde{N_0}+1}}^{+\infty}e^{-\lambda n^\gamma}\mathbb{E}\Big(
|X^\varepsilon_n-X^*_n|^2+|Y^\varepsilon_n-Y^*_n|^2+|Z^\varepsilon_n-Z^*_n|^2\Big)\notag\\
\le& o(1),\quad \varepsilon\to 0.
\end{align}
Particularly, $e^{-\lambda \tilde{N_0}^\gamma}\mathbb{E}\left(|\tilde{Y}_{\tilde{N_0}}^\varepsilon|^2+|\tilde{Z}_{\tilde{N_0}}^\varepsilon|^2\right)=o(1)$, and it follows
\begin{align}\label{3.13}
\sum_{n=0}^{\tilde{N_0}-1}e^{-\lambda n^\gamma}\mathbb{E}\left(|\tilde{Y}_n^\varepsilon|^2+|\tilde{Z}_n^\varepsilon|^2\right)=o(1),\quad \varepsilon\to 0.
\end{align}
Equations (\ref{3.12}) and (\ref{3.13}) complete the proof of equation (\ref{3.9}).
\end{proof}

Then we give the stochastic maximum principle for control problem (\ref{state})-(\ref{costBSDE}), which is the main result of this section.
\begin{theorem}\label{SMP}
Let assumptions (H3.1)-(H3.3) hold, $u^*,X^*$ be the optimal control and the corresponding state process. Then the following inequality holds:
\begin{align*}
\left[b_u^*(n)p_n+\sigma_u^*(n)p_n\sum_{k=0}^{n-1}\gamma(n,k)\xi_k^H+\beta(n,n)\sigma_u^*(n)q_n-f_u^*(n)k_n\right]\cdot(u_n-u^*_n)\ge 0,
\end{align*}
a.s., for all $n\in\mathcal{T}, u\in\mathbb{U}$. Here  $\gamma(n,k)=\sum_{l=0}^{n-1}\beta(n,l)\alpha(l,k)$ and $\alpha(\cdot,\cdot), \beta(\cdot,\cdot)$ are determined functions given by Lemma 2 of \cite{han2025maximum}, and $(p,q,k)$ defined by
\begin{align*}
\left\{\begin{array}{ll}
k_{n+1}=k_n+f_y^*(n)k_n+f_z^*(n)k_n\eta_n,\quad n\ge 1,\\\\
k_1=-1,\quad k_0=0.
\end{array}\right.
\end{align*}
and
\begin{align}\label{BSDEpq}
\quad\left\{\begin{array}{ll}
e^{-\lambda n^{\gamma}}\left(p_n+q_n\eta_n\right)\\\\=e^{-\lambda (n+1)^\gamma}\Big(p_{n+1}+b_x^*(n+1)p_{n+1}+\beta(n+1,n+1)\sigma_x^*(n+1)q_{n+1}\\\\\qquad\qquad\qquad-f_x^*(n+1)k_{n+1}+\sigma_x^*(n+1)p_{n+1}\mathbb{E}^{\mathcal{F}_{_n+1}}[\xi_{n+1}^H]\Big),
\\\\ \lim_{N\to0}\mathbb{E}\left[e^{-\lambda n^\gamma}|p_N|^{2bp_1^{-1}...p_N^{-1}}\right]=0 .
\end{array}\right.
\end{align}
\end{theorem}
\begin{proof}
According to assumptions (H3.1)-(H3.3), it is easy to check BS$\Delta$Es (\ref{BSDEpq}) satisfies (H2.3)-(H2.5), so that there exists a unique solution of BS$\Delta$Es (\ref{BSDEpq}).
Consider
\begin{align}\label{adj1}
&\Delta(e^{-{\lambda n^\gamma}}p_n\cdot\hat{X}_n)\notag\\=&e^{-{\lambda (n+1)^\gamma}}p_{n+1}\cdot\hat{X}_{n+1}-e^{-{\lambda n^\gamma}}p_n\cdot\hat{X}_n\notag\\
=&-e^{-{\lambda (n+1)^\gamma}}\hat{X}_{n+1}\big[b_x^*(n+1)p_{n+1}+b(n+1,n+1)\sigma_x^*(n+1)q_{n+1}\notag\\
&\qquad\qquad\qquad\qquad-f_x^*(n+1)k_{n+1}+\sigma_x^*(n+1)p_{n+1}\mathbb{E}^{\mathcal{F}_{n+1}}[\xi_{n+1}^H]\big]\notag\\
&+e^{-\lambda n^\gamma}\hat{X}_{n+1}q_n\eta_n+e^{-\lambda n^\gamma}p_n\left[b_x^*(n)\hat{X}_n+b_u^*(n)v_n\right]\notag\\
&+e^{-\lambda n^\gamma}p_n\left[\sigma_x^*(n)\hat{X}_n+\sigma_u^*(n)v_n\right]\xi_n^H\notag\\
=&-\left[e^{-\lambda (n+1)^\gamma}b_x^*(n+1)p_{n+1}\hat{X}_{n+1}-e^{-\lambda n^\gamma}b_x^*(n)p_n\hat{X}_n\right]\notag\\
&-\left[e^{-\lambda (n+1)^\gamma}\sigma_x^*(n+1)p_{n+1}\hat{X}_{n+1}\mathbb{E}^{\mathcal{F}_{_n+1}}[\xi_{n+1}^H]-e^{-\lambda n^\gamma}\sigma_x^*(n)p_n\hat{X}_n\xi_n^H\right]\notag\\
&-\left[e^{-\lambda (n+1)^\gamma}b(n+1,n+1)\sigma_x^*(n+1)q_{n+1}\hat{X}_{n+1}-e^{-\lambda n^\gamma}\sigma_x^*(n)q_n\hat{X}_n\eta_n\xi_n^H\right]\notag\\
&+e^{-\lambda n^\gamma}\left[V_n+b^*(n)V_n\right]q_n\eta_n+e^{-\lambda n^\gamma}\sigma_u^*(n)q_nv_n\eta_n\xi_n^H\notag\\
&+e^{-\lambda (n+1)^\gamma}f_x^*(n+1)k_{n+1}\hat{X}_{n+1}+e^{-\lambda n^\gamma}b_u^*(n)p_nv_n+e^{-\lambda n^\gamma}\sigma_u^*(n)p_nv_n\xi_{n}^H.
\end{align}
Notice that 
\begin{align*}
\lim_{N\to+\infty}e^{-\lambda N^\gamma}\mathbb{E}|l_N|^2=0,\quad l=p,q,\hat{X},
\end{align*}
and 
\begin{align*}
\mathbb{E}\left(M_n\eta_n\right)=&\mathbb{E}\left(M_n\mathbb{E}\left[\eta_n|\mathcal{F}_n\right]\right)=0,\\\\
\mathbb{E}\left[M_n\eta_n\xi_n^H\right]=&\mathbb{E}\left[M_n\mathbb{E}\left(\eta_n\xi_n^H|\mathcal{F}_n\right)\right]\notag\\
=&\mathbb{E}\left[M_n\mathbb{E}\left(\eta_n\sum_{k=0}^n\beta(n,k)\eta_k|\mathcal{F}_n\right)\right]\notag\\
=&\mathbb{E}\left[M_n\sum_{k=0}^{n-1}\beta(n,k)\eta_k\mathbb{E}\left(\eta_n|\mathcal{F}_n\right)\right]\notag\\
&+\mathbb{E}\left[M_n\mathbb{E}\left(\beta(n,n)\eta_n^2|\mathcal{F}_n\right)\right]\notag\\
=&\beta(n,n)\mathbb{E}\left[M_n\right],\\\\
\mathbb{E}\left[M_n\xi_n^H\right]=&\mathbb{E}\left[M_n\mathbb{E}\left(\sum_{l=0}^n\beta(n,l)\eta_l^H|\mathcal{F}_n\right)\right]\\
=&\mathbb{E}\left[M_n\mathbb{E}\left(\sum_{l=0}^{n-1}\sum_{k=0}^l\beta(n,l)\alpha(l,k)\xi_k^H|\mathcal{F}_n\right)\right]\\
&+\mathbb{E}\left[M_n\beta(n,n)\mathbb{E}\left(\eta_n|\mathcal{F}_n\right)\right]\\
:=&\mathbb{E}\left[M_n\sum_{k=0}^{n-1}\gamma(n,k)\xi_k^H\right],
\end{align*}
for all $\mathcal{F}_n$-measurable $M_n$. So that, by equation (\ref{adj1}) and combining with $\hat{X}_0=0$, we have
\begin{align}\label{3.16}
0=\sum_{n=0}^{+\infty}e^{-\lambda n^\gamma}\Big[&\beta(n,n)\sigma_u^*(n)q_nv_n+f_x^*(n)k_{n}\hat{X}_{n}\notag\\&+b_u^*(n)p_nv_n+\sigma_u^*(n)p_nv_n\sum_{k=0}^{n-1}\gamma(n,k)\xi_k^H\Big].
\end{align}
Similarly, it is easy to get
\begin{align}\label{3.17}
0=&\sum_{n=1}^{+\infty}\mathbb{E}\Delta(k_n\cdot e^{-\lambda n^\gamma}\hat{Y}_n)+e^{-\lambda}k_1\hat{Y}_1\notag\\
=&-\sum_{n=2}^{+\infty}\mathbb{E}\left[e^{-\lambda n^\gamma}k_n[f_x^*(n)\hat{X}_n+f_u^*(n)v_n)]\right]\\&+e^{-\lambda}\mathbb{E}\left[k_1\cdot[\hat{Y}_1+f_y^*(1)\hat{Y}_1+f_z^*(1)\hat{Z}_1]\right]\notag\\
=&-\sum_{n=0}^{+\infty}\mathbb{E}\left[e^{-\lambda n^\gamma}k_n[f_x^*(n)\hat{X}_n+f_u^*(n)v_n)]\right]\notag\\&+e^{-\lambda}\mathbb{E}\left[k_1\cdot[\hat{Y}_1+f_x^*(1)\hat{X}_1+f_y^*(1)\hat{Y}_1+f_z^*(1)\hat{Z}_1+f_u^*(1)v_1]\right],
\end{align}
since $k_0=0$. 
Combining with equation (\ref{3.16}), (\ref{3.17}), we obtain
\begin{align*}
&\sum_{n=0}^{+\infty}\left[b_u^*(n)p_n+\sigma_u^*(n)p_n\sum_{k=0}^{n-1}\gamma(n,k)\xi_k^H+\beta(n,n)\sigma_u^*(n)q_n-f_u^*(n)k_n\right]\cdot v_n\\
=&e^{-\lambda}\mathbb{E}\left[k_1\cdot[\hat{Y}_1+f_x^*(1)\hat{X}_1+f_y^*(1)\hat{Y}_1+f_z^*(1)\hat{Z}_1+f_u^*(1)v_1]\right]\\
=&\mathbb{E}[\hat{Y}_0+\hat{Z}_0\eta_0]\\
=&\hat{Y}_0.
\end{align*}
Since $u^*$ is the optimal control, we have 
\begin{align*}
\frac{d}{d\varepsilon}J(u^*+\varepsilon v)\Big|_{\varepsilon=0}=\hat{Y}_0\ge 0.
\end{align*}

By the arbitrary of $v$, we have 
\begin{align*}
\mathbb{E}\left\{\mathbf{1}_\mathcal{A}\left[b_u^*(n)p_n+\sigma_u^*(n)p_n\sum_{k=0}^{n-1}\gamma(n,k)\xi_k^H+\beta(n,n)\sigma_u^*(n)q_n-f_u^*(n)k_n\right]\cdot v_n\right\}\ge 0,
\end{align*}
for all $n\in\mathcal{T},\,\mathcal{A}\in\mathcal{F}_n$, which implies 
\begin{align}\label{nec}
\left[b_u^*(n)p_n+\sigma_u^*(n)p_n\sum_{k=0}^{n-1}\gamma(n,k)\xi_k^H+\beta(n,n)\sigma_u^*(n)q_n-f_u^*(n)k_n\right]\cdot(u_n-u^*_n)\ge 0,
\end{align}
a.s. for all $n\in\mathcal{T}$, which completes the proof of Theorem \ref{SMP}.
\end{proof}

\begin{remark}
Indeed, if we define the Hamiltonian function by
\begin{align*}
H(n,x,y,z,u,p,q,k)=&b(n,x,u)p+\sigma(n,x,u) p\mathbb{E}^{\mathcal{F}_n}[\xi_{n+1}]\\&+\beta(n,n)\sigma(n,x,u)q+f(n,x,y,z,u)k,
\end{align*}
where $\mathbb{E}^{\mathcal{F}_n}[\xi_{n+1}]=\sum_{k=0}^{n-1}\gamma(n,k)\xi_k^H$. Then we rewrite the adjoint equation (\ref{BSDEpq}) and the necessary condition for optimal control as 
\begin{align*}
\left\{\begin{array}{ll}
e^{-\lambda n^{\gamma}}\left(p_n+q_n\eta_n\right)=e^{-\lambda (n+1)^{\gamma}}\left(p_{n+1}+H_x^*(n+1)\right),
\\\\ \lim_{N\to0}\mathbb{E}\left[e^{-\lambda n^\gamma}|p_N|^{2bp_1^{-1}...p_N^{-1}}\right]=0,
\end{array}\right.
\end{align*}
and
\begin{align*}
H_u^*(n)\cdot(u_n-u^*_n)\ge0.
\end{align*}
Here
\begin{align*}
H^*(n)=H(n,X^*_n,Y^*_n,Z^*_n,u^*_n,p_n,q_n,k_n).
\end{align*}
\end{remark}
Moreover, if we add convex condition of $H$, we have the following verification theorem.
\begin{theorem}
Assume that $H(n,x,y,z,u,p,q,k)$ is convex w.r.t $(x,y,z,u)$ and condition (\ref{nec}) holds for all $u\in\mathbb{U}$. Then $u^*$ is the optimal control of problem (\ref{state})-(\ref{costBSDE}).
\end{theorem}
\begin{proof}
Let $u$ be any admissible control and $(X,Y,Z), (X^*,Y^*,Z^*)$ be the corresponding processes to $u$, $u^*$, respectively. Denote $(\delta X,\delta Y,\delta Z,\delta u)=(X-X^*,Y-Y^*,Z-Z^*,u-u^*)$ and
\begin{align*}
&\delta l(n)=l(n,X_n,u_n)-l(n,X^*_n,u^*),\quad l=b,\sigma,\\
&\delta f(n)=f(n,X_n,Y_n,Z_n,u_n)-f(n,X^*_n,Y^*_n,Z^*_n,u^*_n).
\end{align*}
It is easy to get
\begin{align*}
0=&\sum_{n=0}^{+\infty}\mathbb{E}\Delta(e^{-\lambda n^\gamma}p_n\cdot\delta X_n)\notag\\
=&-\sum_{n=0}^{+\infty}\mathbb{E}\left[e^{-\lambda (n+1)^\gamma}\delta X_{n+1}H_x^*(n+1)\right]\notag\\
&+\sum_{n=0}^{+\infty}\mathbb{\mathbb{E}}\left[e^{-\lambda n^\gamma}[\delta b(n)p_n+\beta(n,n)\delta\sigma(n)q_n]\right],
\end{align*}
and
\begin{align*}
-e^{-\lambda}k_1\delta Y_1=&\sum_{n=1}^{+\infty}\mathbb{E}\Delta(k_n\cdot e^{-\lambda n^\gamma}\delta Y_n)\\
=&\sum_{n=0}^{+\infty}\mathbb{E}\left[e^{-\lambda n^\gamma}k_n[-\delta f(n)+f_z^*(n)\delta Z_n+f_y^*(n)\delta Y_n]\right]\\
&+e^{-\lambda}k_1\delta f(1).
\end{align*}
So that 
\begin{align*}
J(u)-J(u^*)=&\delta Y_0\\=&-k_1e^{-\lambda}\left(\delta Y_1+\delta f(1)\right)\\
=&-\sum_{n=0}^{+\infty}\mathbb{E}\left[e^{-\lambda (n+1)^\gamma}\delta X_{n+1}H_x^*(n+1)\right]\notag\\
&+\sum_{n=0}^{+\infty}\mathbb{E}\left[e^{-\lambda n^\gamma}[\delta H(n)-H_y^*(n)\delta Y_n-H^*_z(n)\delta Z_n]\right]\\
\ge&-\sum_{n=0}^{+\infty}\mathbb{E}\left[e^{-\lambda (n+1)^\gamma}\delta X_{n+1}H_x^*(n+1)\right]\notag\\
&+\sum_{n=0}^{+\infty}\mathbb{E}\left[e^{-\lambda n^\gamma}[H_x^*(n)\delta X_n+H_u^*(n)\delta u_n]\right]\\
=&H_x^*(0)\delta X_0+\sum_{n=0}^{+\infty}\mathbb{E}\left[e^{-\lambda n^\gamma}H_u^*(n)\delta u_n\right]\\
\ge&0,
\end{align*}
since $\delta X_0=0$. Here
\begin{align*}
\delta H(n)=H(n,X_n,Y_n,Z_n,u_n,p_n,q_n,k_n)-H^*(n).
\end{align*}

\end{proof}

\section{Application} In this section, we consider an optimal investment problem. Let $S_1$ be a risk-free security (e.g. a bond) and $S_2$ be a risky security (e.g. a stock) with following states:
\begin{align*}
\left\{\begin{array}{ll}
S_1(n+1)=\left(1+r_n\right)S_1(n),
\\S_1(0)=s_1,
\end{array}\right.
\end{align*}
and
\begin{align*}
\left\{\begin{array}{ll}
S_2(n+1)=\left(1+\mu_n\right)S_2(n)+\sigma_nS_2(n)\xi^H_n,
\\S_2(0)=s_2.
\end{array}\right.
\end{align*}
Here $r,\mu,\sigma>0$ are determined processes such that $r_n<\mu_n$.  An investor invests a fraction $v_n$ of the funds in risky security at each time $n$, with the remaining portion invested in risk-free security, and consumes a specific proportion $c_n\in(0,1)$ of the funds at some specific time $\mathbb{N}=\{n_1,n_2,...\}$. So the state process is defined as
\begin{align*}
\left\{\begin{array}{ll}
X_{n+1}=\left(1+r_n\right)\left(X_n-c_nX_n\mathcal{X}_{\{n\in\mathbb{N}\}}\right)+(\mu_n-r_n)v_n+\sigma_nv_n\xi^H_n,
\\X_0=x_0.
\end{array}\right.
\end{align*}
Assume that the utility derived by investors from delayed consumption decreases more rapidly as time increases. So we use $e^{-\lambda n^{\gamma}}$ to denote the utility at time $n$ discounted to the present time $0$. Then investor aim to maximize consumption at each time $n_i\in\mathbb{N}$, and minimize the negative utility caused by risky investment at the same time. So the cost function is defined as $J(v)=Y_0$, where $Y_0$ is the solution to 
\begin{align*}
\left\{\begin{array}{ll}
e^{-\lambda n^\gamma}\left(Y_n+Z_n\eta_n\right)=e^{-\lambda (n+1)^\gamma}\left((1+\frac{\lambda}{2})Y_{n+1}-Q_{n+1}X_{n+1}\mathcal{X}_{\{n+1\in\mathbb{N}\}}+R_{n+1}v_{n+1}^\beta\right),
\\\\\lim_{N\to+\infty}\mathbb{E}[e^{-\lambda N^{\gamma}}\|Y_N\|^{2b\delta_1^{-1}...\delta_{N}^{-1}}]=0.
\end{array}\right.
\end{align*}
Here $\beta>1$ and $Q_n,R_n>0$ are determined functions.

According to the obtained results, the adjoint equations are $k_0=0$, $k_n=-(1+\frac{\lambda}{2})^{n-1},\, n\ge 1$, and 
\begin{align}\label{4.1}
\quad\left\{\begin{array}{ll}
e^{-\lambda n^{\gamma}}\left(p_n+q_n\eta_n\right)\\\\=e^{-\lambda (n+1)^\gamma}\Big((1+r_n)\left(p_{n+1}-c_{n+1}p_{n+1}\mathcal{X}_{\{n+1\in\mathbb{N}\}}\right)+Q_{n+1}k_{n+1}\mathcal{X}_{\{n+1\in\mathbb{N}\}}\Big),
\\\\ \lim_{N\to0}\mathbb{E}\left[e^{-\lambda n^\gamma}|p_N|^{2bp_1^{-1}...p_N^{-1}}\right]=0 .
\end{array}\right.
\end{align}
For BS$\Delta$Es (\ref{4.1}), $(p,q)$ can be approximated by
\begin{align*}
\quad\left\{\begin{array}{ll}
e^{-\lambda n^{\gamma}}\left(p_n^N+q_n^N\eta_n\right)\\\\=e^{-\lambda (n+1)^\gamma}\Big((1+r_n)\left(p_{n+1}^N-c_{n+1}p_{n+1}^N\mathcal{X}_{\{n+1\in\mathbb{N}\}}\right)+Q_{n+1}k_{n+1}\mathcal{X}_{\{n+1\in\mathbb{N}\}}\Big),
\\\\ p_N^N=0,
\end{array}\right.
\end{align*}
as we show in Theorem \ref{BSDE e,u}, which directly implies $q_n=0,\,\forall n\in\mathcal{T}$.
Finally, the optimal control is given by
\begin{align*}
v^*_n=\left(0\vee\frac{(\mu_n-r_n)p_n+\sigma_np_n\sum_{k=0}^{n-1}\gamma(n,k)\xi_k^H}{\beta k_nR_n}\right)^{\frac{1}{\beta-1}}\wedge X^*_n(1-c_n\mathcal{X}_{\{n\in\mathbb{N}\}}).
\end{align*}
Set $\mu_n=0.15, r_n=0.05, \sigma_n=0.2, \lambda=1, \beta=2, c_n=0.5, Q=1,R=0.01, \mathbb{N}=\{10,20,...\}$. Take $H=0.75$ and $H=0.25$, respectively. The following are simulation results:
\begin{figure}[H]
    \centering
    \includegraphics[width=0.9\linewidth]{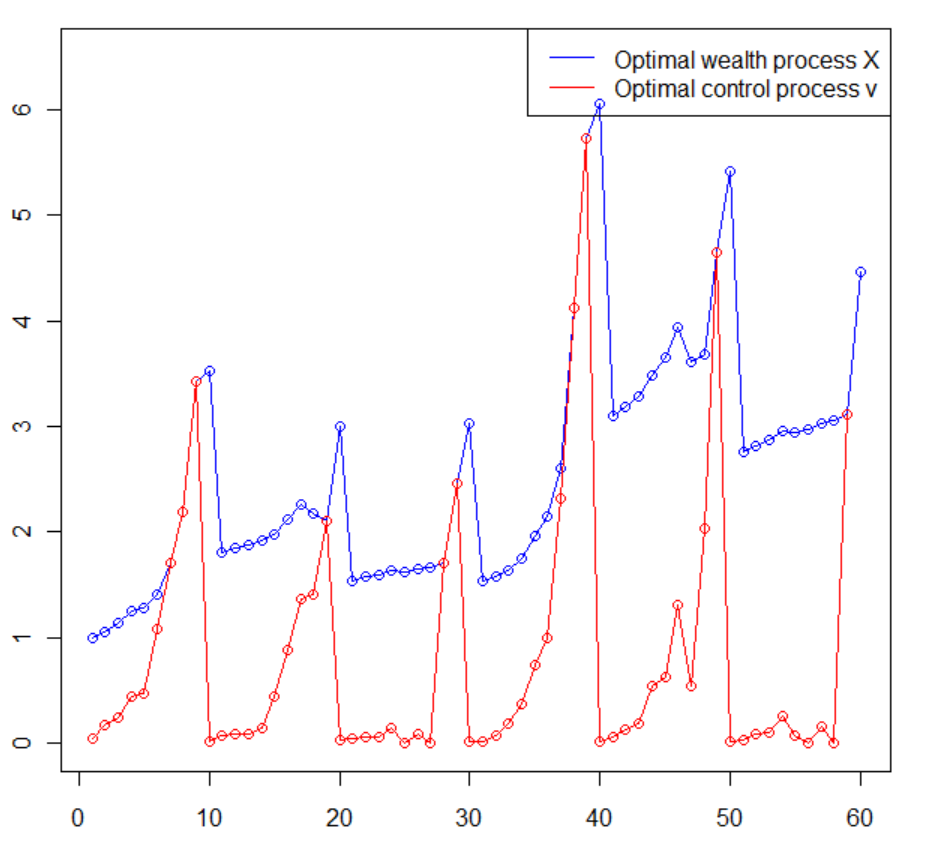}
    \caption{The case driven by fractional noise with $H=0.75$.}
    \label{fig:enter-label}
\end{figure}

\begin{figure}[H]
    \centering
    \includegraphics[width=0.9\linewidth]{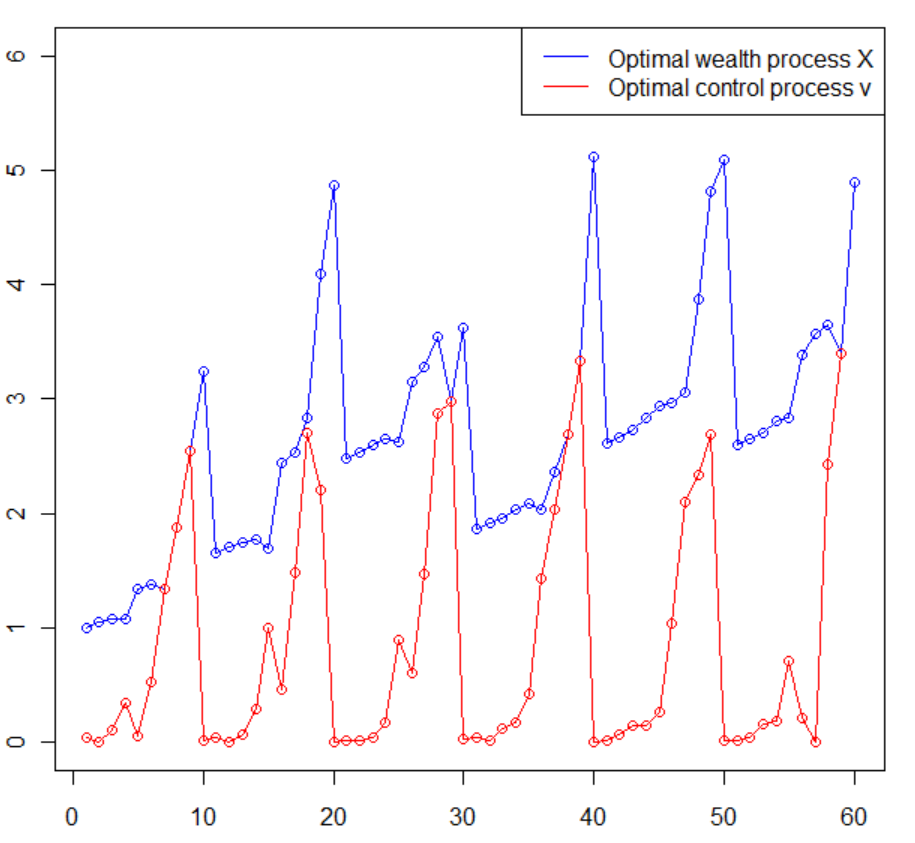}
    \caption{The case driven by fractional noise with $H=0.25$.}
    \label{fig:enter-label}
\end{figure}

\section{Acknowledgments}

This work was supported by National Key R$\&$D Program of China (Grant number 2023YFA1009200) and National Science Foundation of China (Grant numbers 12471417).

\FloatBarrier
\bibliography{main}
\end{document}